\documentclass[11pt]{article} 
\usepackage{latexsym}
\usepackage{amsfonts}
\usepackage{amsthm}
\usepackage{amsmath}
\usepackage{amssymb}
\usepackage{setspace}
\usepackage{eucal}
\usepackage{graphicx}
\usepackage[all, poly, knot]{xy}
\input xy
\xyoption{all}
\xyoption{knot}
\xyoption{arc}

\newtheorem{thm}{Theorem}
\newtheorem{Def}{Definition}
\newtheorem{Prop}[thm]{Proposition}
\newtheorem{Cor}[thm]{Corollary}

\newtheorem{Lemma}[thm]{Lemma}

\begin{document}

\thispagestyle{empty}
 \title{A Sequence of Degree One Vassiliev Invariants for Virtual Knots}
 \author{Allison Henrich}

 \maketitle 
\abstract{
For ordinary knots in $\mathbf{R}^3$, there are no degree one Vassiliev invariants. For virtual knots, however, the space of degree one Vassiliev invariants is infinite dimensional. We introduce a sequence of three degree one Vassiliev invariants of virtual knots of increasing strength. We demonstrate that the strongest invariant is a universal Vassiliev invariant of degree one for virtual knots in the sense that any other degree one Vassiliev invariant can be recovered from it by a certain natural construction. To prove these results, we extend the based matrix invariant introduced by Turaev for virtual strings to the class of singular flat virtual knots with one double-point.}

\markright{Allison Henrich} \vspace{.2cm}    
\section{Introduction}
\subsection{Virtual knots}

One way of representing a knot in $\mathbf{R}^3$ is using a knot diagram, i.e. a closed curve in the plane whose only singularities are a finite number of transversal double-points. These double-points are decorated to indicate which arc represents the over-strand and which represents the under-strand of the knot. Clearly, many different diagrams represent the same knot. Two knot diagrams represent the same knot if and only if they can be related by a sequence of Reidemeister moves~\cite{Reid}, as pictured in Figure~\ref{ReidMoves}. A similar description is used to describe diagrammatic equivalence of links. A more thorough introduction to knots and links can be found in~\cite{Adams}. 

\begin{figure}[h]
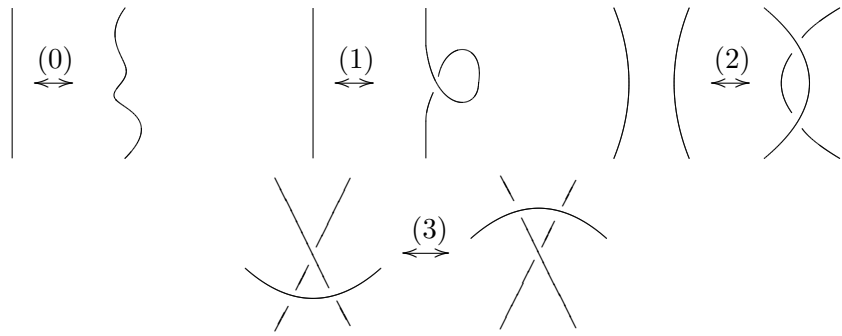

\[
\xy
(-50,-10)*{}; (-50,10)*{} **\dir{-};
{\ar@{<->}(-47,0)*{}; (-42,0)*{}}; ?(.75)*dir{>}+(-1, 3)*{(0)};
(-35,-10)*{}="A";
(-35,10)*{}="B";
"A"; "B" **\crv{(-30,-5) & (-40,-2) & (-32,1) & (-38,6)};
(-10,-10)*{}; (-10,10)*{} **\dir{-};
{\ar@{<->}(-7,0)*{}; (-2,0)*{}}; ?(.75)*dir{>}+(-1, 3)*{(1)};
(5,10)*{}="C";
(5, -10)*{}="D";
(5,5)*{}="C'";
(5,-5)*{}="D'";
"C"; "C'" **\dir{-};
"D"; "D'" **\dir{-};
(8,0)*{}="MB";
(12,0)*{}="LB";
"C'"; "LB" **\crv{(6, -4) & (12,-4)}; \POS?(.25)*{\hole}="2z";
"LB"; "2z" **\crv{(13, 6) & (7, 6)};
"2z"; "D'" **\crv{(5,-3)};
(30,10)*{}="E";
(30, -10)*{}="F";
"E"; "F" **\crv{(34, 0)};
(40,10)*{}="G";
(40, -10)*{}="H";
"G"; "H" **\crv{(36, 0)};
{\ar@{<->}(43,0)*{}; (48,0)*{}}; ?(.75)*dir{>}+(-1, 3)*{(2)};
(50,10)*{}="I";
(50, -10)*{}="J";
"I"; "J" **\crv{(62,0)}; \POS?(.25)*{\hole}="2x"; \POS?(.75)*{\hole}="2y";
(60,10)*{}="K";
(60, -10)*{}="L";
"K"; "2x" **\crv{(55,7)};
"2x"; "2y" **\crv{(50, 0)};
"2y"; "L" **\crv{(55, -7)};
\endxy
\]
\[
\xy
(75,10)*{}="AT";
(85, -10)*{}="AB";
(85, 10)*{}="BT";
(75, -10)*{}="BB";
(71, -2)*{}="CL";
(89, -2)*{}="CR";
"CL"; "CR" **\crv{(80, -10)}; \POS?(.35)*{\hole}="a"; \POS?(.65)*{\hole}="b"; 
"AT"; "b" **\crv{}; \POS?(.35)*{\hole}="c";
"b"; "AB" **\crv{};
"BB"; "a" **\crv{};
"a"; "c" **\crv{};
"c"; "BT" **\crv{};
{\ar@{<->}(92,0)*{}; (98,0)*{}}; ?(.75)*dir{>}+(-1, 3)*{(3)};
(105,10)*{}="A'T";
(115, -10)*{}="A'B";
(115, 10)*{}="B'T";
(105, -10)*{}="B'B";
(101, 2)*{}="C'L";
(119, 2)*{}="C'R";
"C'L"; "C'R" **\crv{(110, 10)}; \POS?(.35)*{\hole}="a'"; \POS?(.65)*{\hole}="b'"; 
"A'T"; "a'" **\crv{};
"a'"; "A'B" **\crv{};\POS?(.65)*{\hole}="c'";
"B'B"; "c'" **\crv{};
"c'"; "b'" **\crv{};
"b'"; "B'T" **\crv{};
\endxy
\]
\caption{Reidemeister moves}\label{ReidMoves}
\end{figure}

For the purposes of this paper, all knots are assumed to be oriented. It is important to note that there are oriented virtual knots that are distinct from the virtual knots with opposite orientations. (This fact was demonstrated by Sawollek~\cite{Sawollek} using a degree one Vassiliev invariant of virtual knots arising from the Conway polynomial.)

Knot diagrams can be encoded by Gauss diagrams, see~\cite{PolyakViro}. A Gauss diagram is an oriented circle parametrizing the knot equipped with signed arrows. The arrows represent crossings in the knot diagram, while the signs together with the direction of the arrows contain information about under- and over-strands. To be more precise, an arrow points from the pre-image of the over-strand of the associated crossing to the pre-image of the under-strand in the parameterizing circle. An arrow is given the sign of its corresponding crossing in the knot diagram. Figure~\ref{Signs} illustrates how these signs are determined in a knot diagram, while Figure~\ref{Gauss} gives an example of a knot diagram and its associated Gauss diagram. 

\begin{figure}[h]
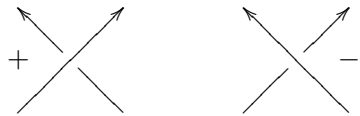

\[
\xy
(-54,-7)*{}; (-40, 7)*{} **\crv{}; \POS?(.5)*{\hole}="y"; ?(1)*\dir{>}+(-14, -7)*{+};
(-40,-7)*{}; "y" **\crv{};
"y"; (-54, 7)*{} **\crv{}; ?(0)*\dir{>}+(-2, -2)*{};
(-10,-7)*{}; (-24, 7)*{} **\crv{}; \POS?(.5)*{\hole}="x"; ?(1)*\dir{>}+(14, -7)*{-};
(-24,-7)*{}; "x" **\crv{};
"x"; (-10, 7)*{} **\crv{}; ?(0)*\dir{>}+(-2, -2)*{};
\endxy
\]
\caption{Sign of a crossing}\label{Signs}
\end{figure}

\def\TrefoilA{\xygraph{!{0;/r2.0pc/:}
!P3"a"{~>{}}
!P9"b"{~:{(1.3288,0):}~>{}}
!P3"c"{~:{(2.5,0):}~>{}}
!{\vover~{"b2"}{"b1"}{"a1"}{"a3"}=>|{1}}
!{"b4";"b2" **\crv{"c1"}}
!{\vover~{"b5"}{"b4"}{"a2"}{"a1"}=>|{3}}
!{"b7";"b5" **\crv{"c2"}}
!{\vover~{"b8"}{"b7"}{"a3"}{"a2"}=>|{2}}
!{"b1";"b8" **\crv{"c3"}}}}

\def\TrefoilGauss{
\begin{xy} /r15mm/:
,{\ellipse<>{}}
,(2,0)="1" ,*+!L{1}
,(.292893, .707107)="3" ,*+!DR{3}
,(1,-1)="2" ,*+!U{2}
,(0,0)="a1" ,*+!LU{+}
,(1.707107, -.707107)="a3" ,*+!D{+}
,(1,1)="a2" ,*+!LU{+}
,{\ar@{->} 0; "1"}
,{\ar@{->} (1,1); (1, -1)}
,{\ar@{<-} (.292893, .707107); (1.707107, -.707107)}
\end{xy}
}

\begin{figure}[h]
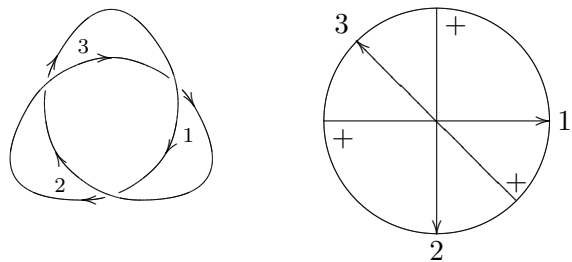

$$\TrefoilA\hspace{10mm}\TrefoilGauss$$
\caption{Example of a knot and its Gauss diagram}\label{Gauss}
\end{figure}

There is a notion of equivalence for Gauss diagrams that corresponds to the equivalence of knot diagrams. In particular, there are combinatorial Reidemeister-type moves that relate Gauss diagrams representing the same knot, see~\cite{PolyakViro}.

Every knot diagram represents an actual knot and can be described by a Gauss diagram, however there exist Gauss diagrams that do not correspond to ordinary knot diagrams. If we allow for an extra type of crossing, called a `virtual crossing' and denoted by an encircled double-point, we can represent any Gauss diagram by a `virtual' knot diagram. These virtual crossings do not give rise to arrows in the Gauss diagram---they are merely an artifact of the virtual knot diagram.  We say that two virtual knot diagrams are \emph{equivalent} if they can be related by a sequence of classical and virtual Reidemeister moves, see Figure~\ref{VirtualReid}. Virtual links are defined similarly. In~\cite{VKT}, Louis Kauffman defined these notions of a virtual knot and a virtual link. He discovered that equivalence classes of virtual knot diagrams can be interpreted as stable equivalence classes of knots in thickened surfaces. This was proven by Carter, Kamada and Saito~\cite{CKS}.

\begin{figure}[h]
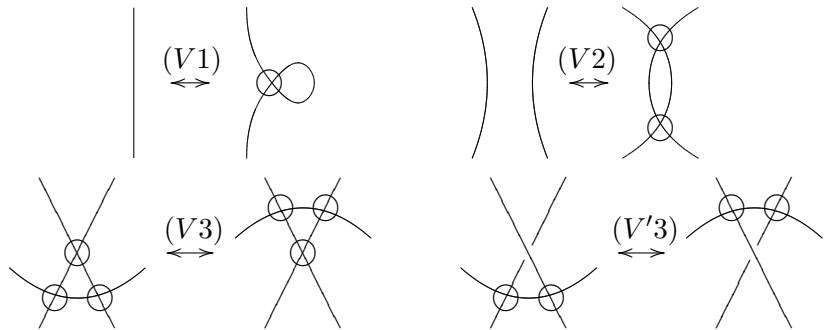

\[
\xy
(-30,-10)*{}; (-30,10)*{} **\dir{-};
{\ar@{<->}(-25,0)*{}; (-20,0)*{}}; ?(.75)*dir{>}+(-1, 3)*{(V1)};
(-15, 10)*{}; (-15, -10)*{} **\crv{(-15,5) & (-12,0) & (-8,-4) & (-5,0) & (-8, 4) & (-12, 0) & (-15, -5)};
(-12,0)*{\bigcirc};
(15,10)*{}="E";
(15, -10)*{}="F";
"E"; "F" **\crv{(19, 0)};
(25,10)*{}="G";
(25, -10)*{}="H";
"G"; "H" **\crv{(21, 0)};
{\ar@{<->}(33,0)*{}; (28,0)*{}}; ?(.75)*dir{>}+(1, 3)*{(V2)};
(35,10)*{}="I";
(35,-10)*{}="J";
(45, 10)*{}="K";
(45, -10)*{}="L";
"I"; "J" **\crv{(40,7) & (43,0) & (40, -7)};
"K"; "L" **\crv{(40, 7) & (37, 0) & (40, -7)};
(40,6)*{\bigcirc};
(40,-6)*{\bigcirc}; 
\endxy
\]
\[
\xy
(75,10)*{}="AT";
(85, -10)*{}="AB";
(85, 10)*{}="BT";
(75, -10)*{}="BB";
(71, -2)*{}="CL";
(89, -2)*{}="CR";
"CL"; "CR" **\crv{(80, -10)};
"AT"; "AB" **\crv{};
"BT"; "BB" **\crv{};
(80,0)*{\bigcirc};
(77,-6)*{\bigcirc};
(83,-6)*{\bigcirc};
{\ar@{<->}(92,0)*{}; (98,0)*{}}; ?(.75)*dir{>}+(-1, 3)*{(V3)};
(105,10)*{}="A'T";
(115, -10)*{}="A'B";
(115, 10)*{}="B'T";
(105, -10)*{}="B'B";
(101, 2)*{}="C'L";
(119, 2)*{}="C'R";
"C'L"; "C'R" **\crv{(110, 10)};
"A'T"; "A'B" **\crv{};
"B'T"; "B'B" **\crv{};
(110,0)*{\bigcirc};
(107,6)*{\bigcirc};
(113,6)*{\bigcirc};
(135,10)*{}="DT";
(145, -10)*{}="DB";
(145, 10)*{}="ET";
(135, -10)*{}="EB";
(131, -2)*{}="FL";
(149, -2)*{}="FR";
"FL"; "FR" **\crv{(140, -10)};
"DT"; "DB" **\crv{}; \POS?(.5)*{\hole}="a";
"ET"; "a" **\crv{};
"a"; "EB" **\crv{};
(137,-6)*{\bigcirc};
(143,-6)*{\bigcirc};
{\ar@{<->}(152,0)*{}; (158,0)*{}}; ?(.75)*dir{>}+(-1, 3)*{(V'3)};
(165,10)*{}="D'T";
(175, -10)*{}="D'B";
(175, 10)*{}="E'T";
(165, -10)*{}="E'B";
(161, 2)*{}="F'L";
(179, 2)*{}="F'R";
"F'L"; "F'R" **\crv{(170, 10)};
"D'T"; "D'B" **\crv{}; \POS?(.5)*{\hole}="a'";
"E'T"; "a'" **\crv{};
"a'"; "E'B" **\crv{};
(167,6)*{\bigcirc};
(173,6)*{\bigcirc};
\endxy
\]
\caption{Virtual Reidemeister moves}\label{VirtualReid}
\end{figure}

\subsection{Virtual homotopy}

Thus far, we have discussed an isotopy equivalence for virtual knots. There is a weaker equivalence relation of virtual knots that is related to the homotopy of curves in thickened surfaces. While homotopy is a trivial notion of equivalence for ordinary knots, it is highly non-trivial for virtual knots.

\begin{Def}  Two virtual knot diagrams are \emph{homotopic} if they are related by a sequence of classical and virtual Reidemeister moves with the additional move that switches the under- and over-strands of a crossing. This additional move is referred to as the \emph{Changing Crossing}, or (CC) move. \end{Def}

The (CC) move exchanges a positive crossing for a negative one and vice versa, and it is equivalent to allowing one strand of a knot to pass through another. We should note that for links, the (CC) move is not permitted when the crossing involves two components. See~\cite{Dye} for more on virtual homotopy of knots and links.  

Any two classical knots  $\mathbf{R}^3$ are clearly homotopic under this definition. As mentioned above, however, there are many non-trivial homotopy classes of virtual knots. A famous virtual knot that is not homotopic to a classical knot is Kishino's knot~\cite{Kishino}, see Figure~\ref{Kishino}.

\begin{figure}[h]
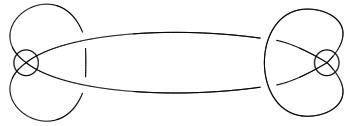

\[
\xy
(-20,0)*{}="a";
(20,0)*{}="b";
"b"; "b" **\crv{(25,5) & (15,10) & (10,0) & (15, -10) & (25,-5)}; 
\POS?(.42)*{\hole}="bt"; \POS?(.58)*{\hole}="bb";
"b"; "bt" **\crv{(15,3)};
"b"; "bb" **\crv{(15,-3)};
"a"; "bt" **\crv{(-15,4) & (0,5)}; \POS?(.333)*{\hole}="at";
"a"; "bb" **\crv{(-15,-4) & (0,-5)}; \POS?(.333)*{\hole}="ab";
"at"; "ab" **\crv{(-12,0)};
"a"; "at" **\crv{(-27,5) & (-15,13)};
"a"; "ab" **\crv{(-27,-5) & (-15,-13)};
(-20,0)*{\bigcirc};
(20,0)*{\bigcirc};
\endxy
\]
\caption{Kishino's knot}\label{Kishino}
\end{figure}

\subsection{Flat virtual knots}

Intuitively, a flat virtual knot diagram is a virtual knot diagram where the under- and over-strand information of each classical crossing is forgotten. Flat virtual knot diagrams are often thought of as shadows of virtual knot diagrams. An example of a flat virtual knot diagram, the shadow of Kishino's knot, is pictured in Figure~\ref{KishinoFlat}.

\begin{figure}[h]
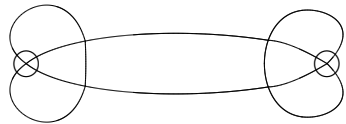

\[
\xy
(-20,0)*{}="a";
(20,0)*{}="b";
"b"; "b" **\crv{(25,5) & (15,10) & (10,0) & (15, -10) & (25,-5)}; 
\POS?(.42)*{}="bt"; \POS?(.58)*{}="bb";
"b"; "bt" **\crv{(15,3)};
"b"; "bb" **\crv{(15,-3)};
"a"; "bt" **\crv{(-15,4) & (0,5)}; \POS?(.333)*{}="at";
"a"; "bb" **\crv{(-15,-4) & (0,-5)}; \POS?(.333)*{}="ab";
"at"; "ab" **\crv{(-12,0)};
"a"; "at" **\crv{(-27,5) & (-15,13)};
"a"; "ab" **\crv{(-27,-5) & (-15,-13)};
(-20,0)*{\bigcirc};
(20,0)*{\bigcirc};
\endxy
\]
\caption{A non-trivial flat virtual knot}\label{KishinoFlat}
\end{figure}

Just as knots are defined as equivalence classes of knot diagrams, flat virtual knots are defined as equivalence classes of flat virtual knot diagrams. Here, the equivalence relation is generated by the shadows of the classical and virtual Reidemeister moves. For pictures of the flat Reidemeister moves, see Figure~\ref{FlatMoves}. We define the flat virtual Reidemeister moves in a similar fashion. (Note that (V'3) is the only virtual move that differs in the flat category.)

\begin{figure}[h]
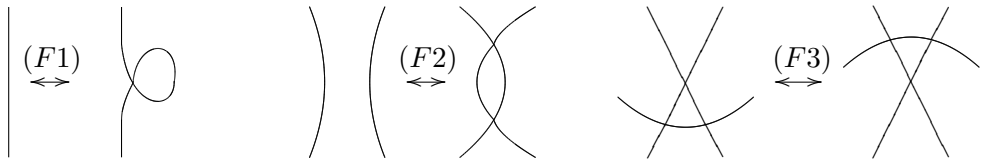

\[
\xy
(-10,-10)*{}; (-10,10)*{} **\dir{-};
{\ar@{<->}(-7,0)*{}; (-2,0)*{}}; ?(.75)*dir{>}+(-1, 3)*{(F1)};
(5,10)*{}="C";
(5, -10)*{}="D";
(5,5)*{}="C'";
(5,-5)*{}="D'";
"C"; "C'" **\dir{-};
"D"; "D'" **\dir{-};
(8,0)*{}="MB";
(12,0)*{}="LB";
"C'"; "LB" **\crv{(6, -4) & (12,-4)}; \POS?(.25)*{}="2z";
"LB"; "2z" **\crv{(13, 6) & (7, 6)};
"2z"; "D'" **\crv{(5,-3)};
(30,10)*{}="E";
(30, -10)*{}="F";
"E"; "F" **\crv{(34, 0)};
(40,10)*{}="G";
(40, -10)*{}="H";
"G"; "H" **\crv{(36, 0)};
{\ar@{<->}(43,0)*{}; (48,0)*{}}; ?(.75)*dir{>}+(-1, 3)*{(F2)};
(50,10)*{}="I";
(50, -10)*{}="J";
"I"; "J" **\crv{(62,0)}; \POS?(.25)*{}="2x"; \POS?(.75)*{}="2y";
(60,10)*{}="K";
(60, -10)*{}="L";
"K"; "2x" **\crv{(55,7)};
"2x"; "2y" **\crv{(50, 0)};
"2y"; "L" **\crv{(55, -7)};
(75,10)*{}="AT";
(85, -10)*{}="AB";
(85, 10)*{}="BT";
(75, -10)*{}="BB";
(71, -2)*{}="CL";
(89, -2)*{}="CR";
"CL"; "CR" **\crv{(80, -10)}; \POS?(.35)*{}="a"; \POS?(.65)*{}="b"; 
"AT"; "b" **\crv{}; \POS?(.35)*{}="c";
"b"; "AB" **\crv{};
"BB"; "a" **\crv{};
"a"; "c" **\crv{};
"c"; "BT" **\crv{};
{\ar@{<->}(92,0)*{}; (98,0)*{}}; ?(.75)*dir{>}+(-1, 3)*{(F3)};
(105,10)*{}="A'T";
(115, -10)*{}="A'B";
(115, 10)*{}="B'T";
(105, -10)*{}="B'B";
(101, 2)*{}="C'L";
(119, 2)*{}="C'R";
"C'L"; "C'R" **\crv{(110, 10)}; \POS?(.35)*{}="a'"; \POS?(.65)*{}="b'"; 
"A'T"; "a'" **\crv{};
"a'"; "A'B" **\crv{};\POS?(.65)*{}="c'";
"B'B"; "c'" **\crv{};
"c'"; "b'" **\crv{};
"b'"; "B'T" **\crv{};
\endxy
\]
\caption{Flat Reidemeister moves}\label{FlatMoves}
\end{figure}

Defined similarly, flat virtual links are equivalence classes of shadows of virtual link diagrams. Any flat virtual link diagram that is the shadow of a classical link diagram is equivalent to a disjoint union of unknots in the flat category. In particular, the shadow of a diagram of a classical knot (possibly containing virtual crossings) is equivalent to the unknot in the flat category. The converse, however, is not true. There are non-classical, non-trivial virtual knots whose shadows are trivial in the flat category.

Because the (CC) virtual homotopy move allows one to switch between positive and negative crossings, it is not hard to see that a flat virtual knot is actually just the homotopy class of a virtual knot. For links, however, the two notions differ. In the category of flat virtual links, (F2) moves are allowed between components. For virtual links, the (CC) move is not permitted when the two strands of a crossing are from different components, so an (F2)-like move isn't always possible. Thus, the notion of homotopy is more rigid than that of flat equivalence.

Flat virtual knots were studied in a paper by Turaev~\cite{Turaev}, who calls them virtual strings. Turaev defines virtual strings as equivalence classes of (unsigned) arrow diagrams. Given a virtual knot diagram and its corresponding Gauss diagram, one can recover the corresponding flat virtual knot diagram by projecting the knot onto its shadow. Meanwhile, the virtual string can be recovered from the Gauss diagram by reversing all negative arrows and dropping the signs. It is easy to describe the relation between a virtual string and its flat virtual knot. In fact, Reidemeister moves for flat virtual knots correspond to moves that may be applied to an arrow diagram of a virtual string. These virtual string moves are related to the equivalence relations on Gauss diagrams via the correspondence mentioned above. We will discuss virtual strings in greater depth in Section~\ref{SBMsection}.

\markright{Allison Henrich} \vspace{.2cm}     

\section{Singular Virtual Knots and Vassiliev Invariants}
\subsection{Singular virtual knots}
We define a sequence of three distinct Vassiliev invariants of virtual knots. The strongest of these invariants is the universal Vassiliev invariant of degree one for virtual knots. 

\begin{Def} \begin{enumerate}
\item A \emph{singular virtual knot (or link)} is a virtual knot (resp. link) with a finite number of \emph{double-points}, i.e. places where the knot (resp. link) has a transverse self-intersection (resp. intersection) point. Equivalently, a singular virtual knot is considered to be an equivalence class of virtual knot (resp. link) diagrams with an additional type of crossing, see Figure~\ref{CrossingTypes}. The equivalence relation on singular virtual knot (resp. link) diagrams is generated by Reidemeister moves, virtual Reidemeister moves, and the moves involving singular double-points given in Figure~\ref{SingEquiv}. 
\item A \emph{flat singular virtual knot} is a singular virtual knot with no distinction made between the two types of classical crossings, positive and negative. Both types of classical crossings are represented in flat singular virtual knot diagrams by their shadows, and are distinct from both virtual and singular crossings. Flat singular virtual knot diagrams are equivalent if and only if they are related by a sequence of flat classical and virtual Reidemeister moves and flat versions of the singularity moves shown in Figure~\ref{SingEquiv}.
\end{enumerate}
\end{Def}

\begin{figure}[h]
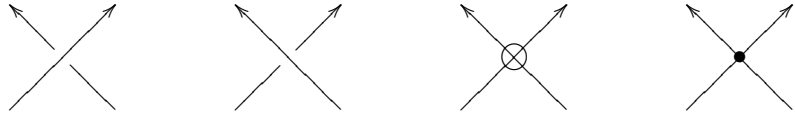

\[
\xy
(-54,-7)*{}; (-40, 7)*{} **\crv{}; \POS?(.5)*{\hole}="y"; ?(1)*\dir{>}+(-14, -7)*{};
(-40,-7)*{}; "y" **\crv{};
"y"; (-54, 7)*{} **\crv{}; ?(0)*\dir{>}+(-2, -2)*{};
(-10,-7)*{}; (-24, 7)*{} **\crv{}; \POS?(.5)*{\hole}="x"; ?(1)*\dir{>}+(-14, -7)*{};
(-24,-7)*{}; "x" **\crv{};
"x"; (-10, 7)*{} **\crv{}; ?(0)*\dir{>}+(-2, -2)*{};
(6,-7)*{}; (20, 7)*{} **\crv{}; ?(0)*\dir{>}+(-2, 3)*{};
(20, -7)*{}; (6, 7)*{} **\crv{}; ?(0)*\dir{>}+(-2, 3)*{};
(13,0)*{\bigcirc};
(36,-7)*{}; (50, 7)*{} **\crv{}; ?(0)*\dir{>}+(-2, 3)*{};
(50, -7)*{}; (36, 7)*{} **\crv{}; ?(0)*\dir{>}+(-2, 3)*{};
(43,0)*{\bullet};
\endxy
\]
\caption{Types of crossings: positive, negative, virtual, and singular}\label{CrossingTypes}
\end{figure}

\begin{figure}[h]
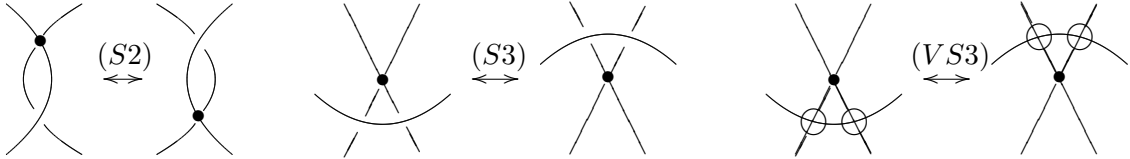

\[
\xy 
(30,10)*{}="E";
(30, -10)*{}="F";
"E"; "F" **\crv{(42,0)}; \POS?(.25)*{\bullet}="2u"; \POS?(.75)*{\hole}="2v";
(40,10)*{}="G";
(40, -10)*{}="H";
"G"; "2u" **\crv{(35,7)};
"2u"; "2v" **\crv{(30, 0)};
"2v"; "H" **\crv{(35, -7)};
{\ar@{<->}(43,0)*{}; (48,0)*{}}; ?(.75)*dir{>}+(-1, 3)*{(S2)};
(50,10)*{}="I";
(50, -10)*{}="J";
(60,10)*{}="K";
(60, -10)*{}="L";
"K"; "L" **\crv{(48,0)}; \POS?(.25)*{\hole}="2x"; \POS?(.75)*{\bullet}="2y";
"I"; "2x" **\crv{(55, 7)};
"2x"; "2y" **\crv{(60, 0)};
"2y"; "J" **\crv{(55, -7)}; 
(75,10)*{}="AT";
(85, -10)*{}="AB";
(85, 10)*{}="BT";
(75, -10)*{}="BB";
(71, -2)*{}="CL";
(89, -2)*{}="CR";
"CL"; "CR" **\crv{(80, -10)}; \POS?(.35)*{\hole}="a"; \POS?(.65)*{\hole}="b"; 
"AT"; "b" **\crv{}; \POS?(.35)*{\bullet}="c";
"b"; "AB" **\crv{};
"BB"; "a" **\crv{};
"a"; "c" **\crv{};
"c"; "BT" **\crv{};
{\ar@{<->}(92,0)*{}; (98,0)*{}}; ?(.75)*dir{>}+(-1, 3)*{(S3)};
(105,10)*{}="A'T";
(115, -10)*{}="A'B";
(115, 10)*{}="B'T";
(105, -10)*{}="B'B";
(101, 2)*{}="C'L";
(119, 2)*{}="C'R";
"C'L"; "C'R" **\crv{(110, 10)}; \POS?(.35)*{\hole}="a'"; \POS?(.65)*{\hole}="b'"; 
"A'T"; "a'" **\crv{};
"a'"; "A'B" **\crv{};\POS?(.65)*{\bullet}="c'";
"B'B"; "c'" **\crv{};
"c'"; "b'" **\crv{};
"b'"; "B'T" **\crv{};
(135,10)*{}="VAT";
(145, -10)*{}="VAB";
(145, 10)*{}="VBT";
(135, -10)*{}="VBB";
(131, -2)*{}="VCL";
(149, -2)*{}="VCR";
"VCL"; "VCR" **\crv{(140, -10)}; \POS?(.35)*{\bigcirc}="va"; \POS?(.65)*{\bigcirc}="vb"; 
"VAT"; "vb" **\crv{}; \POS?(.35)*{\bullet}="vc";
"vc"; "VAB" **\crv{};
"vc"; "VBB" **\crv{};
"vb"; "VAB" **\crv{};
"VBB"; "va" **\crv{};
"va"; "vc" **\crv{};
"vc"; "VBT" **\crv{};
{\ar@{<->}(152,0)*{}; (158,0)*{}}; ?(.75)*dir{>}+(-1, 3)*{(VS3)};
(165,10)*{}="VA'T";
(175, -10)*{}="VA'B";
(175, 10)*{}="VB'T";
(165, -10)*{}="VB'B";
(161, 2)*{}="VC'L";
(179, 2)*{}="VC'R";
"VC'L"; "VC'R" **\crv{(170, 10)}; \POS?(.35)*{\bigcirc}="va'"; \POS?(.65)*{\bigcirc}="vb'"; 
"VA'T"; "va'" **\crv{};
"va'"; "VA'B" **\crv{};\POS?(.65)*{\bullet}="vc'";
"vc'"; "VA'T" **\crv{};
"vc'"; "VB'T" **\crv{};
"VB'B"; "vc'" **\crv{};
"vc'"; "vb'" **\crv{};
"vb'"; "VB'T" **\crv{};
\endxy
\]
\caption{Equivalence relations for singular double-points}\label{SingEquiv}
\end{figure}

It is possible to discuss what it means for singular virtual knots to be invariant under homotopy. Two such knots are homotopic if they are related by a sequence of the Reidemeister-type moves for singular virtual knots, taken in conjunction with the Changing Crossing move. Just as with flat virtual knots and links, flat singular virtual knots correspond to homotopy classes of singular virtual knots. On the other hand, the notion of invariance under homotopy for singular virtual links is more restrictive than flat equivalence for singular virtual links, as the (CC) move is never allowed when a crossing involves more than one component of the link. 

\subsection{Vassiliev invariants}
Any virtual knot invariant $V$ with values in an abelian group $A$ can be extended to an invariant of singular virtual knots. To define $V$ on a singular virtual knot, we define the \emph{$n$th derivative of $V$}, an invariant of singular virtual knots with $n$ double-points via the following recursive relation. 

\begin{equation}\label{nthderiv}V^{(n)}(K)=V^{(n-1)}(K^+)-V^{(n-1)}(K^-)\end{equation}

In the equation, let $n\geq 1$ and suppose that $K$ is a singular knot (diagram) with $n$ double-points. Choosing a particular double-point in $K$, we obtain $K^{+}$ by resolving the double-point into a positive crossing, and we obtain $K^{-}$ by resolving the double-point into a negative crossing. To define the first derivative, we set $V^{(0)}=V$. Let us view the derivative of an invariant as a recursive definition for the invariant $V$ of a singular virtual knot, $K$, with $n$ double-points. We see that the derivative gives us a value in $A$ for $V(K)$ that does not depend on the order in which we resolve the $n$ double-points.

\begin{Def} A \emph{Vassiliev invariant of degree $\leq n$ for virtual knots} is an invariant of virtual knots that vanishes on singular virtual knots with more than $n$ double-points. The smallest such $n$ is called the degree of the invariant. \end{Def}

This definition was originally introduced by Kauffman in~\cite{VKT}. We note that the theory of Vassiliev invariants of virtual knots resulting from this definition differs from the one proposed by Goussarov, Polyak and Viro in~\cite{PolyakViro}. In particular, Kauffman's theory of degree one Vassiliev invariants is highly non-trivial, while in the theory of Vassiliev invariants introduced in~\cite{PolyakViro}, there are no Vassiliev invariants of degree one or two.

A \emph{universal} Vassiliev invariant of degree one, say $\mathbf{U}_1$, has the following additional property. Any other Vassiliev invariant of degree one with values in an abelian group can be recovered from $\mathbf{U}_1$ by a certain natural construction. To be more precise, we consider the following definitions.

Let $V$ be a Vassiliev invariant of degree one for virtual knots, taking values in an abelian group $A$. Denote by $V^{(1)}$ the \emph{first derivative of $V$}, as in~\eqref{nthderiv}. Note that the map $V^{(1)}$ is constant on homotopy classes of singular virtual knots, since $V$ is Vassiliev of degree one. Now let $\mathbf{Z}[\mathcal{H}^1]$ be the free abelian group generated by the set of homotopy classes of singular virtual knots with one double-point.  We have an induced map $V^{\ast}:\mathbf{Z}[\mathcal{H}^1]\rightarrow A$ that is defined on $K_0\in\mathcal{H}^1$ as follows and extended to all of $\mathbf{Z}[\mathcal{H}^1]$ by linearity.

\begin{equation}\label{Vstar}V^{\ast}([K_0])=V^{(1)}(K_0)
\end{equation}

This map is well-defined since $V^{(1)}$ is constant within each homotopy class of singular virtual knots. We may now define the universal Vassiliev invariant of degree one and give a sufficient condition for universality.

\begin{Def} A Vassiliev invariant, $\mathbf{U}_1$, of degree one for virtual knots is the \emph{universal Vassiliev invariant of degree one} if every other Vassiliev invariant $V$ of degree one for virtual knots can be recovered from $\mathbf{U}_1$ using only the first derivative of $V$ and the values of $V$ on one representative for each homotopy class of virtual knots. For instance, for all virtual knots $K$ homotopic to $K_0$, $\mathbf{U}_1$ is universal if, given \emph{any} degree one Vassiliev invariant $V$ of virtual knots, $\mathbf{U}_1$ satisfies the following equation. $$V(K)=V(K_0)+V^{\ast}\left(\frac{1}{2}(\mathbf{U}_1(K)-\mathbf{U}_1(K_0))\right)$$
\end{Def}

\section{Invariants of Virtual Knots}
In this section, we give a sequence of three new degree one Vassiliev invariants of virtual knots. The strongest of the three is a universal Vassiliev invariant of degree one.
\subsection{The polynomial invariant}

For a given virtual knot $K$ with diagram $\widetilde K$, let us choose a classical crossing, $d$, in this diagram and perform the smoothing operation pictured in Figure~\ref{Smooth}. 

\begin{figure}[h]
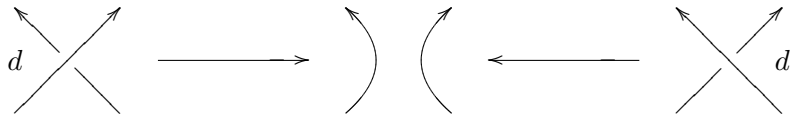

\[
\xy
(-24,-7)*{}; (-10, 7)*{} **\crv{}; \POS?(.5)*{\hole}="y"; ?(1)*\dir{>}+(-14, -7)*{d};
(-10,-7)*{}; "y" **\crv{};
"y"; (-24, 7)*{} **\crv{}; ?(0)*\dir{>}+(-2, -2)*{};
{\ar@{->}(-5,0)*{}; (15,0)*{}}; ?(.75)*\dir{-}+(-5, 3)*{};
(20,-7)*{}; (20, 7)*{} **\crv{(28,0)}; ?(1)*\dir{>}+(-2, 3)*{};
(34, -7)*{}; (34, 7)*{} **\crv{(26,0)}; ?(1)*\dir{>}+(-2, 3)*{};
{\ar@{<-}(39,0)*{}; (59,0)*{}}; ?(.75)*\dir{-}+(-5, 3)*{};
(78,-7)*{}; (64, 7)*{} **\crv{}; \POS?(.5)*{\hole}="z"; ?(1)*\dir{>}+(14, -7)*{d};
(64,-7)*{}; "z" **\crv{};
"z"; (78, 7)*{} **\crv{}; ?(0)*\dir{>}+(-2, -2)*{};
\endxy
\]
\caption{Smoothing of a crossing}\label{Smooth}
\end{figure}

This smoothing produces a two-component virtual link. We choose an ordering $(1,2)$ for the components of the link, and let $1\cap 2$ denote the set of crossings between the two components. Taking the flat virtual link that is the shadow of the smoothed virtual link, we give signs to each classical crossing in $1\cap 2$ as shown in Figure~\ref{Index}. 

\begin{figure}[h]
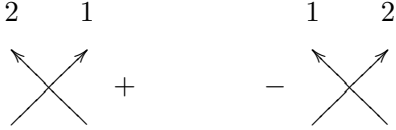

\[
\xy
(-20,-5)*{}; (-10, 5)*{} **\crv{}; ?(0)*\dir{>}+(0, 5)*{1};
(-10,-5)*{}; (-20, 5)*{} **\crv{}; ?(0)*\dir{>}+(0, 5)*{2};
(20,-5)*{}; (30, 5)*{} **\crv{}; ?(0)*\dir{>}+(0, 5)*{2};
(30, -5)*{}; (20, 5)*{} **\crv{}; ?(0)*\dir{>}+(0, 5)*{1};
(-5,0)*{+}; (15,0)*{-};
\endxy
\]
\caption{Definition of sgn(x) for $x\in 1\cap 2$}\label{Index}
\end{figure}

Now we have the following definition.

\begin{Def} The \emph{intersection index} of an ordered virtual link associated to a smoothed crossing $d$ in a virtual knot diagram $\widetilde K$ is given by $$i(d)=\sum _{x\in 1\cap 2} sgn(x).$$
\end{Def}

This intersection index is useful for distinguishing virtual knots because of the following property.

\begin{Lemma} Let $L$ be an ordered two-component flat virtual link, and denote by $1\cap 2$ the set of crossings between components. The sum $i(L)=\sum _{x\in 1\cap 2} sgn(x)$, where $sgn(x)$ is defined in Figure~\ref{Index}, is invariant under flat virtual equivalence.
\end{Lemma}.

It is an easy exercise to verify that the sum $i(L)$ is invariant under all classical and virtual flat Reidemeister moves. Furthermore, reordering the components of the link merely changes the sign of the intersection index. Thus, we define a polynomial invariant for virtual knots as follows.

\begin{Def} Let the polynomial $\mathbf{p}_t(K)\in\mathbf{Z}[t]$ for virtual knot $K$ with diagram $\widetilde K$ be as follows.
$$\mathbf{p}_t(K)=\sum _d sign(d)(t^{|i(d)|}-1)$$ 
In this formula, the sum is over all non-virtual crossings, $d$, in $\widetilde K$, the quantity $sign(d)$ is the sign of the crossing $d$ as defined in Figure~\ref{Signs}, and the letter $t$ is a variable. 
\end{Def} 

We take the absolute value of the intersection index to account for the ambiguity in the choice of ordering of the components of the smoothed virtual link. It should be noted that this polynomial is similar to the $u_{\pm}$ cobordism invariants of knots in thickened surfaces introduced by Turaev in~\cite{TuraevCobordism}. Now we show that our polynomial has the following property.

\begin{Prop}\label{polyprop} The polynomial $\mathbf{p}_t$ is a Vassiliev invariant of degree one for virtual knots. \end{Prop}
\begin{proof} First, let us verify that $\mathbf{p}_t$ is indeed an invariant. Consider the behavior of the polynomial under Reidemeister move 1. This move adds a new crossing, and therefore adds a new term to the sum. If we smooth at the new crossing, we obtain a two-component virtual link with disjoint components. Thus, $|i(d)|=0$ for the new crossing $d$. Hence, the term corresponding to the new crossing vanishes. The terms corresponding to the other crossings are unchanged since the Reidemeister 1 move does not affect the intersection index of curves arising from other crossings. Next, we consider Reidemeister 2. This move adds two crossings to a diagram and, hence, adds two new terms to the sum. It is clear that the absolute value of the intersection index of the link obtained by smoothing at the first new crossing is equal to the absolute value of the intersection index of the link corresponding to smoothing at the second crossing. Since these crossings have opposite signs, our two new terms cancel each other. The remaining terms are unchanged since Reidemeister moves applied to links do not affect intersection indices of flat links arising from other crossings. It is an easy exercise to verify that the terms of the sum defining $\mathbf{p}_t$ are unchanged under a Reidemeister 3 move. Finally, since our sum is over classical crossings and since virtual Reidemeister moves don't effect intersection indices of flat links, $\mathbf{p}_t$ is invariant under all virtual Reidemeister moves as well. Thus, $\mathbf{p}_t$ is an invariant of virtual knots.

To prove that $\mathbf{p}_t$ is Vassiliev of degree one, let $K_{ab}$ be a singular virtual knot with two double-points, $a$ and $b$. Now let $K_{a+b+}$ denote the knot obtained by resolving both $a$ and $b$ positively, and define $K_{a+b-}$, $K_{a-b+}$, and $K_{a-b-}$ similarly. Then using the recursion relation~\eqref{nthderiv} from the definition of Vassiliev invariant, we have the following.
$$\mathbf{p}_t(K_{ab})=\mathbf{p}_t(K_{a+b+})-\mathbf{p}_t(K_{a+b-})-\mathbf{p}_t(K_{a-b+})+\mathbf{p}_t(K_{a-b-}).$$
Note that, for any $\epsilon ,\delta =+,-$,all of the terms in each polynomial $\mathbf{p}_t(K_{a\epsilon b\delta })$ are the same except for the terms coming from the new $a$ and $b$ crossings. Furthermore, $|i(a)|$ and $|i(b)|$ are the same in each polynomial $\mathbf{p}_t(K_{a\epsilon b\delta })$. The only difference is that two terms $(t^{|i(a)|}-1)$ have coefficient +1 and two have coefficient -1, while two $(t^{|i(b)|}-1)$ terms have +1 coefficient and two have coefficient -1. Hence, $\mathbf{p}_t(K_{ab})=0$. It is an easy exercise to find a singular virtual knot $K$ with one double-point such that the first derivative of $\mathbf{p}_t$ doesn't vanish on $K$. (Consider, for instance, the example in Figure~\ref{virt_tref}. If we glue at one of the upper crossings, we get a singular virtual knot with one double-point such that resolving one way gives us a virtual knot with non-zero $\mathbf{p}_t$ value, while resolving the other way gives the unknot.) Hence, we conclude that $\mathbf{p}_t$ is Vassiliev of degree one.
\end{proof}

\textbf{Remark} It is interesting to note that the virtual knot in Figure~\ref{virt_tref} has trivial Jones polynomial. (See~\cite{VKT} for the definition of the Jones polynomial for virtual knots.) However, a quick computation shows that, for this ``virtualized trefoil" $K$, $\mathbf{p}_t(K)=2(t^2-1)$.

\begin{figure}[h]
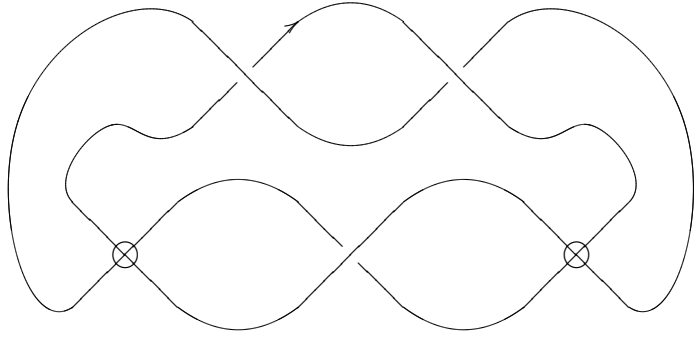

\[
\xy
(-54,-7)*{}; (-40, 7)*{} **\crv{}; 
(-40,-7)*{}; (-54, 7)*{} **\crv{}; 
(-47,0)*{\bigcirc};
(-24,-7)*{}; (-10, 7)*{} **\crv{}; \POS?(.5)*{\hole}="x"; 
(-10,-7)*{}; "x" **\crv{};
"x"; (-24, 7)*{} **\crv{}; 
(6,-7)*{}; (20, 7)*{} **\crv{}; 
(20, -7)*{}; (6, 7)*{} **\crv{}; 
(13,0)*{\bigcirc};
(-24,17)*{}; (-38, 31)*{} **\crv{}; \POS?(.5)*{\hole}="y"; 
(-38,17)*{}; "y" **\crv{};
"y"; (-24,31)*{} **\crv{};
(4,17)*{}; (-10, 31)*{} **\crv{}; \POS?(.5)*{\hole}="z"; 
(-10,17)*{}; "z" **\crv{};
"z"; (4,31)*{} **\crv{};
(-24,31)*{}; (-10,31)*{} **\crv{(-17,36)}; ?(0)*\dir{>}+(0, 5)*{};
(-24,17)*{}; (-10,17)*{} **\crv{(-17,12)};
(-40,-7)*{}; (-24,-7)*{}; **\crv{(-32,-13)};
(-40,7)*{}; (-24,7)*{}; **\crv{(-32,13)};
(-10,-7)*{}; (6,-7)*{}; **\crv{(-2,-13)};
(-10,7)*{}; (6,7)*{}; **\crv{(-2,13)};
(-54, 7)*{}; (-38,17)*{} **\crv{(-57,10)&(-47,23)&(-44,13)};
(20, 7)*{}; (4, 17)*{} **\crv{(23,10)&(13,23)&(10,13)};
(-54, -7)*{}; (-38,31)*{} **\crv{(-60,-11)&(-70,23)&(-45,36)};
(20, -7)*{}; (4, 31)*{} **\crv{(26,-11)&(36,23)&(11,36)};
\endxy
\]
\caption{The virtualized trefoil}\label{virt_tref}
\end{figure}

Our new polynomial invariant has yet another use, as illustrated by the following proposition. 

\begin{Prop}\label{mod2} Let $\mathbf{p}_t(mod2)$ be the polynomial invariant with coefficients reduced mod 2. Then $\mathbf{p}_t(mod2)$ is a homotopy invariant of virtual knots.\end{Prop}
\begin{proof} Since $\mathbf{p}_t$ is invariant under classical and virtual Reidemeister moves, so is $\mathbf{p}_t(mod2)$. Furthermore, we note that performing a (CC) move on a crossing in a diagram simply changes the coefficient of the term in $\mathbf{p}_t$ associated to that crossing from $\pm1$ to $\mp1$. Thus, the polynomial with coefficients mod 2 is also invariant under the (CC) move.
\end{proof}

\subsection{The smoothing invariant}
We continue building our sequence of degree one Vassiliev invariants by introducing an analog of $\mathbf{p}_t$ that uses the definition of flat virtual links.

\begin{Def} Let $K$ be a virtual knot with diagram $\widetilde K$. Let $\widetilde K_{smooth}^d$ be the unordered two-component virtual link obtained by smoothing $\widetilde K$ at the crossing $d$ in $\widetilde K$, as illustrated in Figure~\ref{Smooth}. Denote by $[\widetilde K_{smooth}^d]$ the equivalence class of the associated flat virtual link. Furthermore, let $[\widetilde K_{link}^0]$ denote the flat equivalence class of the union of the shadow of $K$ with an unlinked copy of the unknot. Then $\mathbf{S}(K)$ is given by the following formula.
$$\mathbf{S}(K)=\sum _d sign(d)([\widetilde K_{smooth}^d]-[\widetilde K_{link}^0])$$
Here, the sum ranges over all non-virtual crossings $d$ in $\widetilde K$, while the value $sign(d)$, the sign of $d$, is $\pm 1$, as defined in Figure~\ref{Signs}.\end{Def}

 \begin{Prop} $\mathbf{S}$ is a Vassiliev invariant of degree one. \end{Prop}
 \begin{proof} The proof of this proposition is similar to the proof of Proposition~\ref{polyprop}.
 \end{proof}
 
\begin{Prop}\label{smooth_strong} The smoothing invariant $\mathbf{S}$ is strictly stronger than the polynomial invariant $\mathbf{p}_t$. \end{Prop}
\begin{proof} The absolute value of the intersection index $|i(d)|$, defined above, is an invariant of two-component flat virtual links. Thus, $|i(d)|$ can be recovered from $[\widetilde K_{smooth}^d]$. Hence, $\mathbf{p}_t$ can be recovered from $\mathbf{S}$. Below we show that for homotopic knots $K_1$ and $K_2$ given in Figure~\ref{knotpair}, $\mathbf{p}_t(K_1)=\mathbf{p}_t(K_2)\text{ but }\mathbf{S}(K_1)\neq \mathbf{S}(K_2)$.

The intersection index of the link obtained by smoothing at the labeled positive crossing in $K_1$ is the same as the intersection index of the link obtained by smoothing at the negative crossing in $K_1$. Thus, the corresponding terms in $\mathbf{p}_t(K_1)$ cancel. Similarly, the terms corresponding to the labeled positive and negative crossings in $\mathbf{p}_t(K_2)$ cancel. So it is easy to see that $\mathbf{p}_t(K_1)=\mathbf{p}_t(K_2)$. 

It remains to show that $\mathbf{S}(K_1)\neq \mathbf{S}(K_2)$. A straightforward analysis of the quantity $\mathbf{S}(K_1)- \mathbf{S}(K_2)$ shows that this difference is equal to $2[K_{1+}]-2[K_{1-}]$, where $[K_{1+}]$ is the flat class of the virtual link obtained by smoothing $K_1$ at the labeled positive crossing and $[K_{1-}]$ is the flat class of the virtual link obtained by smoothing $K_1$ at the labeled negative crossing. We'd like to show that these two flat virtual links are distinct. To do this, we make use of a map $B$ that is an invariant of (ordered) two-component flat virtual links. The invariant, defined as follows, is closely related to the Goldman Lie bracket.

$$B(K)=\sum _{x\in 1\cap 2}sgn(x)[K_x]$$

The value $sgn(x)$ is defined in Figure~\ref{Index}, and $K_x$ is the flat virtual knot obtained by smoothing $K$ at a crossing $x\in 1\cap 2$, that is, at a crossing involving both components of the link. We again use square brackets to denote the flat equivalence class of the knot. Now, computing $B$ on our two links (with some arbitrary choice of ordering of the components), we will see that $B(K_{1+})\neq\pm B(K_{1-})$.

To distinguish the flat virtual knots in $B(K_{1+})$ from those in $B(K_{1-})$, we make use of the invariant defined in Proposition~\ref{mod2}. We showed that the polynomial invariant with coefficients mod 2 was a homotopy invariant of virtual knots. Since homotopy classes of virtual knots coincide with flat classes of virtual knots, the invariant can also be thought of as an invariant of flat virtual knots. Straightforward computations of $\mathbf{p}_t(mod2)$ allow us to easily distinguish terms in $B(K_{1+})$ from those in $B(K_{1-})$.
\end{proof}

\begin{figure}[h]
\hspace{.5in}
\includegraphics[height=1.3in]{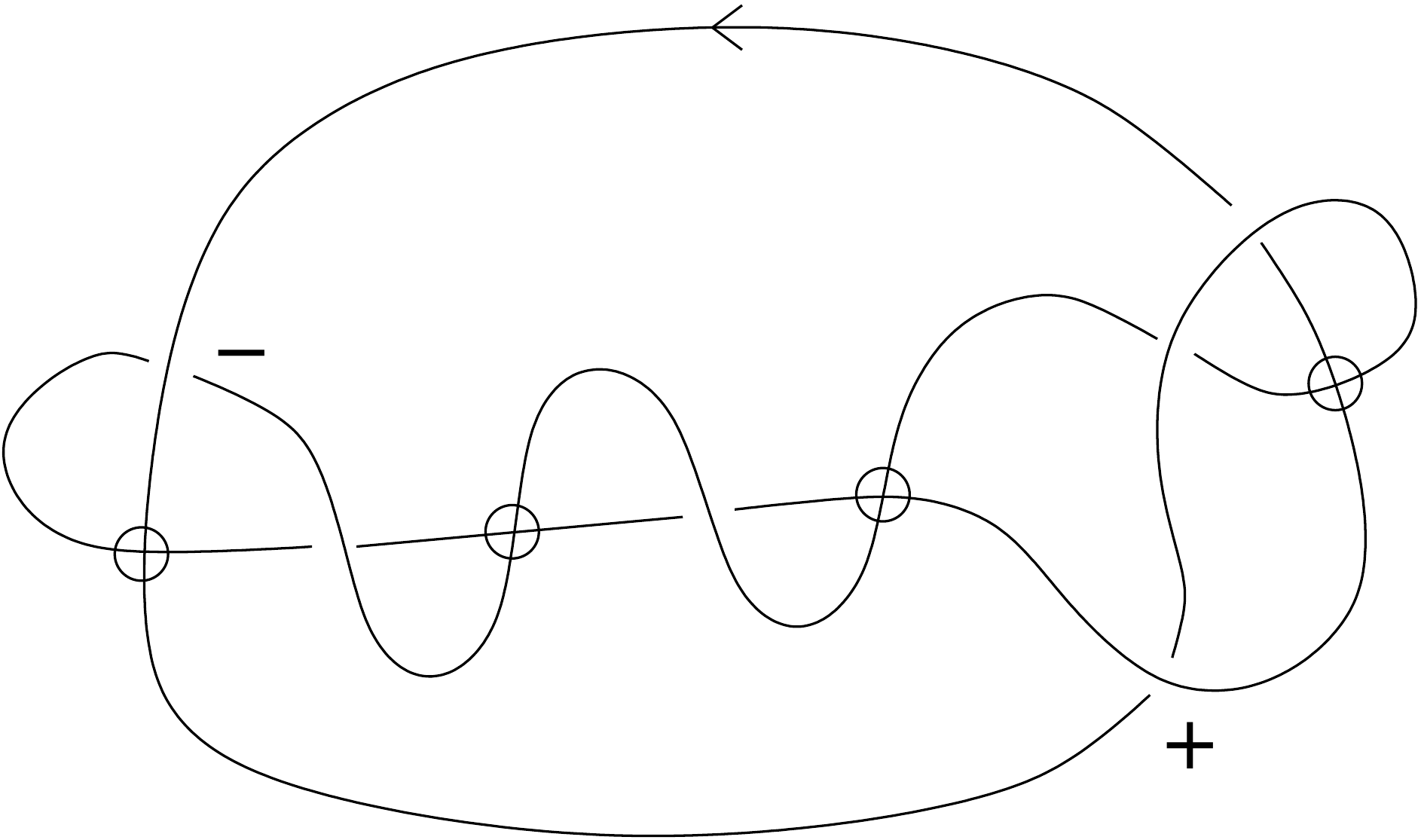}
\hspace{.8in}
\includegraphics[height=1.3in]{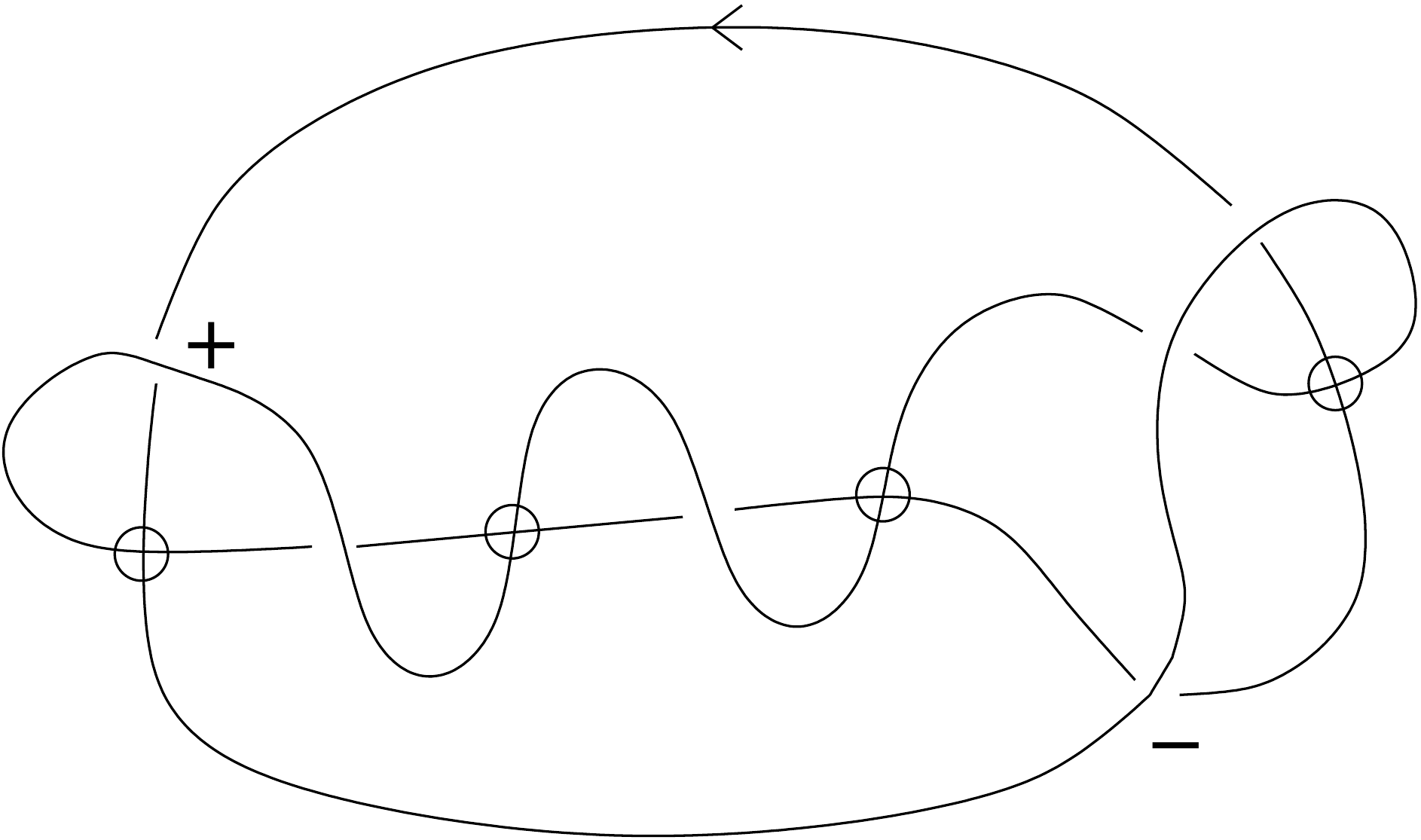}
$$K_1\hspace{6cm}K_2$$
\caption{Example illustrating that $\mathbf{S}$ is strictly stronger than $\mathbf{p}_t$}\label{knotpair}
\end{figure}

 \subsection{The gluing invariant}
 
As we will discover shortly, the next invariant is even stronger than the previous two. Instead of counting intersection indices or creating flat classes of virtual links, we use the flat singular virtual knots introduced earlier in this paper.

\begin{Def} Suppose that $K$ is a virtual knot and choose a fixed diagram, $\widetilde K$, of $K$. Denote by $[\widetilde K_{glue}^d]$ the flat equivalence class of the shadow of the singular virtual knot obtained by gluing strands of $\widetilde K$ together at the crossing $d$ to produce a double-point. Let $[\widetilde K_{sing}^0]$ represent the flat class of the singular virtual knot with one double-point obtained by introducing a small kink into $\widetilde K$ (via a Reidemeister 1 move) and gluing at the kink crossing. (Note: the resulting flat equivalence class is independent of where in the diagram the kink is initially placed.) Now let $\mathbf{G}(K)$ be the virtual knot invariant given by the following formula.
$$\mathbf{G}(K)=\sum _d sign(d)([\widetilde K_{glue}^d]-[\widetilde K_{sing}^0])$$
Again, the sum ranges over all classical crossings $d$ in $\widetilde K$, while the value $sign(d)$, the sign of $d$, is $\pm 1$, as defined in Figure~\ref{Signs}.\end{Def}

\begin{thm} $\mathbf{G}$ is the universal Vassiliev invariant of degree one for virtual knots. \end{thm}
\begin{proof} A straightforward combinatorial argument shows that $\mathbf{G}$ is indeed an invariant, while the proof that $\mathbf{G}$ is Vassiliev of degree one is similar to the proofs for $\mathbf{p}_t$ and $\mathbf{S}$. It remains to show that $\mathbf{G}$ is the \emph{universal} Vassiliev invariant of degree one. 

Let $V$ be a degree one Vassiliev invariant of virtual knots, and let $K\in[K_0]$, where $K$ is a virtual knot in the homotopy class of the representative virtual knot $K_0$. Since $K$ and $K_0$ are homotopic, they have diagrams that are related by a sequence of Reidemeister moves and $m$ (CC) moves, where $m\geq0$. Let $K_0,K_1,...,K_{m-1},K_m=K$ be a sequence of knots such that $K_i$ and $K_{i-1}$ have diagrams that differ by exactly one (CC) move. Then we have the following.

\begin{eqnarray} V(K) - V(K_0) & = & V(K_m)-V(K_{m-1})+V(K_{m-1})-\cdots -V(K_1)+V(K_1)-V(K_0)\nonumber\\
&=& \sum _{i=1}^m (V(K_i)-V(K_{i-1}))\nonumber\\
&=& \sum _{i=1}^m sign(i)V^{(1)}(K_i \bullet)\nonumber
\end{eqnarray}
Here, let $K_i \bullet$ denote the knot obtained by gluing at the crossing in the diagram of $K_i$ that differs from the one in the diagram of $K_{i-1}$ by a (CC) move. If resolving $K_i \bullet$ positively gives the diagram for $K_i$, set $sign(i)=1$. If, instead, resolving $K_i \bullet$ positively gives $K_{i-1}$, set $sign(i)=-1$. Then, by equation~\eqref{Vstar}, we have the following.

\begin{eqnarray}
V(K)-V(K_0) &=& V^{\ast}(\sum _{i=1}^m sign(i)[K_i \bullet])\nonumber\\
&=& V^{\ast}\left(\frac{1}{2}(\mathbf{G}(K)-\mathbf{G}(K_0))\right)\nonumber
\end{eqnarray}

This last equality follows directly from the definition of $\mathbf{G}$. We note that $\mathbf{G}(K)-\mathbf{G}(K_0)$ takes values in $2\mathbf{Z}[\mathcal{H}^1]$, since (CC) moves change the coefficients of terms from $\pm 1$ to $\mp 1$. This explains why the $\frac{1}{2}$ appears in the last equality. Now by the definition of universality, the proof that $\mathbf{G}$ is universal is complete.
\end{proof}

One interesting result is that, when a knot diagram is smoothed at some crossing to create a link, more information is lost than if the knot is glued at the crossing. The next theorem illustrates this.

\begin{thm}\label{stronger} $\mathbf{G}$ is strictly stronger than $\mathbf{S}$, i.e. the following two properties hold. \begin{enumerate} 
\item If a pair of homotopic virtual knots, $K_i$ and $K_j$ satisfy $\mathbf{S}(K_i)\neq \mathbf{S}(K_j)$, then  $\mathbf{G}(K_i)\neq \mathbf{G}(K_j)$. 
\item There exist two homotopic virtual knots $K_1$ and $K_2$ such that $\mathbf{S}(K_1)= \mathbf{S}(K_2)$ but  $\mathbf{G}(K_1)\neq \mathbf{G}(K_2)$.\end{enumerate}\end{thm}

The first result simply follows from the universality of $\mathbf{G}$. To prove the second result, we extend Turaev's based matrix invariant of flat virtual knots to the class of flat singular virtual knots with one double-point. We shall leave this task for section~\ref{SBMsection}.

\section{Virtual Strings and Their Matrices}\label{SBMsection}

\subsection{Virtual strings revisited}

In~\cite{Turaev}, Vladimir Turaev introduced the notion of virtual strings. The results Turaev proved for virtual strings are particularly useful in our setting as they directly correspond to flat virtual knots. Virtual strings are defined using arrow diagrams rather than knot diagrams. Thus, it is not surprising that their correspondence to flat virtual knots is analogous to the correspondence between Gauss diagrams and virtual knots. The difference is that Turaev's virtual strings contain arrows that aren't signed, and the direction of the arrow has a new meaning. To be more precise, we consider the following definitions.

\begin{Def}[Turaev] \begin{enumerate} 
\item For an integer $m\geq 0$, a \emph{virtual string $\alpha$ of rank m} is an oriented circle $S$, called the \emph{core circle} of $\alpha$, and a distinguished set of $2m$ distinct points of $S$ partitioned into ordered pairs. 
\item These $m$ ordered pairs of points are called \emph{arrows} of $\alpha$, and the collection of all arrows of $\alpha$ is denoted $arr(\alpha)$. 
\item The endpoints $a,b\in S$ of an arrow $(a,b)\in arr(\alpha)$ are called its \emph{tail} and \emph{head}, respectively.
\item Two virtual strings are \emph{homeomorphic} if there is an orientation-preserving homeomorphism of the core circles, transforming the set of arrows of the first string into the set of arrows of the second string. By abuse of language, we will refer to homeomorphism classes of virtual strings as virtual strings.
\end{enumerate}\end{Def}

Given a virtual string $\alpha$, we can realize the corresponding flat virtual knot diagram as follows. Each arrow in the string is associated with a crossing in the knot diagram. Passing through the head of an arrow while traveling along the oriented core circle $S$ is equivalent to passing through the strand of a crossing indicated by the letter $b$ in the diagram below. Passing through the tail of an arrow while traveling along the core circle $S$ is equivalent to passing through a crossing on the strand $a$, pictured below.\\

\begin{figure}[h]
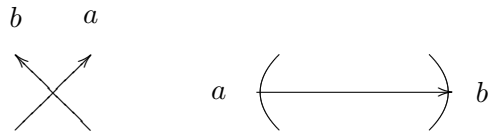

\[
\xy
(-20,-5)*{}; (-10, 5)*{} **\crv{}; ?(0)*\dir{>}+(0, 5)*{a};
(-10,-5)*{}; (-20, 5)*{} **\crv{}; ?(0)*\dir{>}+(0, 5)*{b};
(15,-5)*{}; (15, 5)*{} **\crv{(10,0)};
(35, -5)*{}; (35, 5)*{} **\crv{(40,0)};
{\ar@{->} (12,0)*{}; (38,0)*{}};
(15,0)*{}; ?(0)*\dir{}+(-5, 0)*{a};
(35,0)*{}; ?(0)*\dir{}+(30, 0)*{b};
\endxy
\]
\caption{A flat crossing and its corresponding string arrow}\label{arrow}
\end{figure}

In Figure~\ref{Kishino}, we gave a diagram for Kishino's knot. In Figure~\ref{KishinoFlat}, we showed the corresponding flat virtual knot. Here, we picture the Gauss diagram for the virtual knot in Figure~\ref{Kishino} and the virtual string that corresponds to the flat virtual knot in Figure~\ref{KishinoFlat}.

\def\KGauss{
\begin{xy} /r15mm/:
,{\ellipse<>{}}
,(1.173648, .984808)="1" ,*+!D{+}
,(.826352,.984808)="2" ,*+!D{-}
,(1.173648, -.984808)="a1" ,*+!U{+}
,(.826352,-.984808)="a2" ,*+!U{-}
,(1.984808, .173648)="3"
,(.015192, .173648)="a3" 
,(1.984808, -.173648)="4"
,(.015192, -.173648)="a4"
,{\ar@{<-} "1"; "a3"}
,{\ar@{->} "2"; "a4"}
,{\ar@{->} "3"; "a1"}
,{\ar@{<-} "4"; "a2"}
\end{xy}
}

\def\KString{
\begin{xy} /r15mm/:
,{\ellipse<>{}}
,(1.173648, .984808)="1" 
,(.826352,.984808)="2" 
,(1.173648, -.984808)="a1" 
,(.826352,-.984808)="a2" 
,(1.984808, .173648)="3"
,(.015192, .173648)="a3" 
,(1.984808, -.173648)="4"
,(.015192, -.173648)="a4"
,{\ar@{<-} "1"; "a3"}
,{\ar@{<-} "2"; "a4"}
,{\ar@{->} "3"; "a1"}
,{\ar@{->} "4"; "a2"}
\end{xy}
}

\begin{figure}[h]
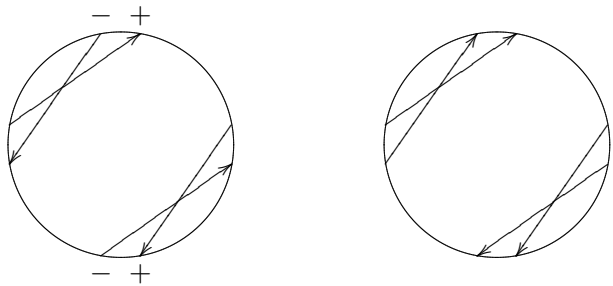

$$\KGauss\hspace{20mm}\KString$$
\caption{Gauss diagram for Kishino's knot and its corresponding virtual string}\label{GaussString}
\end{figure}

There is another notion of equivalence that is useful for virtual strings, namely \emph{homotopy}. The equivalence relation that defines the homotopy of virtual strings is designed to model the Reidemeister-type equivalence relation of flat virtual knots. While homeomorphisms of virtual strings correspond to ambient isotopies of knot diagrams, homotopy moves for virtual strings correspond to the moves for flat virtual knots that (i) add a kink in the knot, (ii) push two strands of the knot past each other to introduce 2 new double-points, and (iii) push a strand of the knot past a double-point.\\

Let us discuss how the moves (i), (ii) and (iii) translate to virtual strings. Suppose $\alpha$ is a virtual string with core circle $S$. We choose two points $a,b\in S$ such that the arc $ab$ is disjoint from $arr(\alpha)$. The move (i) adds an arrow $(a,b)$ to $arr(\alpha)$.\\

 Now choose two arcs on $S$ that are disjoint from $arr(\alpha)$, the first with endpoints $a$ and $a'$ and the second with endpoints $b$ and $b'$ (in any order). The move (ii) adds two new arrows, $(a,b)$ and $(b',a')$. \\
 
 Finally, let $a,a',b,b',c,c'\in S$ such that $(a',b),(b',c),(c',a)$ are arrows of $\alpha$ and the arcs $aa',bb',cc'$ are all disjoint from other arrows. The move (iii) replaces $(a',b),(b',c),(c',a)$ with $(a,b'),(b,c'),(c,a')$. \\

We have seen that homotopy classes of virtual strings can be thought of as equivalence classes of flat virtual knots. Thus, homotopy invariants of virtual strings will prove to be quite useful in studying flat virtual knots. In particular, invariants of virtual strings will help us to evaluate the virtual knot invariants introduced in the last section. One particularly useful invariant virtual strings is introduced in the next section.

\subsection{Based matrices}

In~\cite{Turaev}, Vladimir Turaev introduces objects called based matrices and their homology classes. To each virtual string, Turaev associates one of these matrices. He showed that if two virtual strings are homotopic, then their associated based matrices are homologous. Let's recall Turaev's definitions from~\cite{Turaev}.\\

\begin{Def} A \emph{based matrix} is a triple $(G, s, b: G\times G \rightarrow H)$. Here, $H$ is an abelian group, $G$ is a finite set, and $s \in G$. The map $b$ is skew-symmetric, i.e. for all $g,h\in G$, $b(g,h)=-b(h,g)$.  \end{Def}

\begin{Def} \begin{enumerate}
\item An element $g \in G -  \left\{s\right\} $ is \emph{annihilating} if $b(g, h)=0$ for all $h \in G$.
\item An element $g \in G -  \left\{s\right\} $ is a \emph{core element} if $b(g, h)=b(s,h)$ for all $h \in G$.
\item Two elements $g_1, g_2 \in G-  \left\{s\right\}$ are \emph{complementary} if $b(g_1,h)+b(g_2,h)=b(s,h)$ for all $h \in G$.
\end{enumerate}
\end{Def}

There are three operations related to these special elements that can be applied to based matrices. These operations, called \emph{elementary extensions}, are as follows.

\begin{enumerate}
\item $M_1$ transforms $(G,s,b)$ into the based matrix $(G_1=G\amalg \{g\},s, b_1)$ such that $b_1:G_1 \times G_1 \rightarrow H$ extends $b$ and $b_1(g,h)=0$ for all $h\in G_1$.
\item $M_2$ transforms $(G,s,b)$ into the based matrix $(G_2=G\amalg \{g\},s, b_2)$ such that $b_2:G_2 \times G_2 \rightarrow H$ extends $b$ and $b_2(g,h)=b_2(s,h)$ for all $h\in G_2$.
\item $M_3$ transforms $(G,s,b)$ into the based matrix $(G_3=G\amalg \{g_i,g_j\},s, b_3)$ such that $b_3:G_3 \times G_3 \rightarrow H$ is any skew-symmetric map extending $b$ with $b_3(g_i,h)+b_3(g_j,h)=b_3(s,h)$ for all $h\in G_3$.
\end{enumerate}

Given the definitions of the three elementary extensions, we can say what it means for a based matrix to be primitive. We can also define Turaev's two notions of equivalence for based matrices.

\begin{Def} \begin{enumerate} 
\item A based matrix is \emph{primitive} if it cannot be obtained from another matrix by an elementary extension. 
\item Two based matrices $(G, s, b)$ and $(G',s',b')$ are \emph{isomorphic} if there is a bijection $G\rightarrow G'$ sending $s$ to $s'$ and transforming $b$ into $b'$.
\item Two based matrices are \emph{homologous} if one can be obtained from the other by a finite number of elementary extensions $M_1^{\pm 1},M_2^{\pm 1},M_3^{\pm 1}$.
\end{enumerate} \end{Def}

Let us discuss the third definition. Since being homologous is an equivalence relation, it is natural to look at the equivalence classes it produces. Turaev proved in~\cite{Turaev} that each of these equivalence classes has a canonical representative, namely the unique (up to isomorphism) primitive based matrix contained in the class. 

\begin{Lemma}[Turaev] Every based matrix is obtained from a primitive based matrix by elementary extensions. Two homologous primitive based matrices are isomorphic.\end{Lemma}

This lemma is particularly useful because it allows us to extend any isomorphism invariant of based matrices to a homology invariant of based matrices. To be precise, the value of the invariant on a based matrix is defined to be the value of the invariant on a homologous primitive based matrix. For example, one of the simplest of these invariants is the number $\# (G)$, the cardinality of any primitive based matrix $G$ in a given homology class.\\

We mentioned at the beginning of the section that based matrices can be associated to virtual strings. Suppose we have a virtual string $\alpha$. Let $G=G(\alpha)$ be the set $\{s\}\amalg arrow(\alpha)$. The map $b=b(\alpha):G\times G \rightarrow \mathbf{Z}$ is defined via intersection indices of certain curves obtained from $\alpha$. These intersection indices can be computed combinatorially. First, given $e\in arrow(\alpha)$, $b(e,s)$ is obtained by smoothing the flat virtual knot associated to $\alpha$ in the direction of the orientation at the crossing corresponding to $e$. Then $b(e,s)$ is the intersection index, $i(e)$, of the right-hand curve with the left-hand curve of the smoothed pair, see Figure~\ref{right}. \\

\begin{figure}[h]
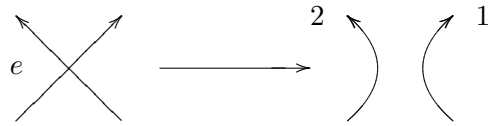

\[
\xy
(-24,-7)*{}; (-10, 7)*{} **\crv{}; \POS?(.5)*{}="y"; ?(1)*\dir{>}+(-14, -7)*{e};
(-10,-7)*{}; "y" **\crv{};
"y"; (-24, 7)*{} **\crv{}; ?(0)*\dir{>}+(-2, -2)*{};
{\ar@{->}(-5,0)*{}; (15,0)*{}}; ?(.75)*\dir{-}+(-5, 3)*{};
(20,-7)*{}; (20, 7)*{} **\crv{(28,0)}; ?(1)*\dir{>}+(-2, 3)*{}; ?(1)*\dir{>}+(-4,0)*{2};
(34, -7)*{}; (34, 7)*{} **\crv{(26,0)}; ?(1)*\dir{>}+(-2, 3)*{}; ?(1)*\dir{>}+(4,0)*{1};
\endxy
\]
\caption{The ordered flat virtual link obtained from smoothing at $e$}\label{right}
\end{figure}

For $e,f \neq s$ in $G$, the number $b(e,f)$ requires a bit more explanation. Viewing $e$ and $f$ as arrows in $\alpha$, suppose $e=(a,b)$ and $f=(c,d)$. Let $(ab)^{\circ}$ be the interior of the arc $ab$ and let $(cd)^{\circ}$ be the interior of the arc $cd$. Define $ab\cdot cd$ as the number of arrows of $\alpha$ with tails in $(ab)^{\circ}$ and heads in $(cd)^{\circ}$ minus the number of arrows with tails in $(cd)^{\circ}$ and heads in $(ab)^{\circ}$. Now in Figure~\ref{linking}, we illustrate what it means for $e$ and $f$ to be linked positively and negatively as well as unlinked. Let $\epsilon =1$ if $f$ links $e$ positively, $\epsilon = -1$ if $f$ links $e$ negatively and $\epsilon =0$ if $e$ and $f$ are unlinked. Finally, we define $b(e,f)=ab\cdot cd+\epsilon$.\\

\def\LinkPos{
\begin{xy} /r15mm/:
,{\ellipse<>{}}
,(2,0)="1" ,*+!L{e}
,(1,-1)="2"
,(0,0)="a1" 
,(1,1)="a2" ,*+!D{f}
,{\ar@{->} 0; "1"}
,{\ar@{<-} (1,1); (1, -1)}
\end{xy}
}

\def\LinkNeg{
\begin{xy} /r15mm/:
,{\ellipse<>{}}
,(2,0)="1" ,*+!L{e}
,(1,-1)="2"
,(0,0)="a1" 
,(1,1)="a2" ,*+!D{f}
,{\ar@{->} 0; "1"}
,{\ar@{->} (1,1); (1, -1)}
\end{xy}
}

\def\LinkNone{
\begin{xy} /r15mm/:
,{\ellipse<>{}}
,(1.173648, .984808)="1" ,*+!D{e}
,(.826352,.984808)="2"
,(1.173648, -.984808)="a1" 
,(.826352,-.984808)="a2" ,*+!U{f}
,{\ar@{->} "a1"; "1"}
,{\ar@{->} "2"; "a2"}
\end{xy}
}

\begin{figure}[h]
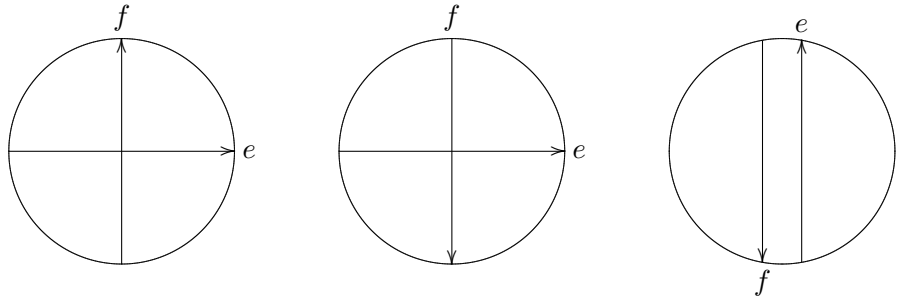

$$\LinkPos\hspace{10mm}\LinkNeg\hspace{10mm}\LinkNone$$
\caption{$f$ links $e$ positively, $f$ links $e$ negatively, and $f$ and $e$ are unlinked}\label{linking}
\end{figure}

As mentioned above, associating matrices to virtual strings in this fashion is a well-defined operation. In~\cite{Turaev}, Turaev proves the following result.

\begin{Lemma}[Turaev] If two virtual strings are homotopic, then their based matrices are homologous. \end{Lemma}

\subsection{Singular based matrices}

To help distinguish equivalence classes of flat singular virtual knots with one double-point, we extend the definition of based matrix introduced by Turaev. Below, we define singular based matrices and their homology classes. 

\begin{Def}  A \emph{singular based matrix}, or SBM, is a quadruple $(G, s, d, b: G\times G \rightarrow H)$. Here, $G$ is a finite set with $s$ and $d$ elements of $G$. The map $b$ is skew-symmetric, i.e. for all $g,h\in G$, $b(g,h)=-b(h,g)$.  \end{Def}

\begin{Def} \begin{enumerate}
\item An element $g \in G -  \left\{s, d\right\} $ is \emph{annihilating} if $b(g, h)=0$ for all $h \in G$.
\item An element $g \in G -  \left\{s, d\right\} $ is a \emph{core element} if $b(g, h)=b(s,h)$ for all $h \in G$.
\item Two elements $g_1, g_2 \in G-  \left\{s, d\right\}$ are \emph{complementary} if $b(g_1,h)+b(g_2,h)=b(s,h)$ for all $h \in G$.
\item We call a distinguished element $g\in\left\{s,d\right\}$ \emph{annihilating-like} if $b(g,h)=0$ for all $h\in G$.
\item We call $d$ \emph{core-like} if $b(d,h)=b(s,h)$ for all $h\in G$.
\end{enumerate}
\end{Def}

As in Turaev's definition, we introduce operations on SBMs. Three of these, $\widetilde M_1,\widetilde M_2,\widetilde M_3$ are \emph{elementary extensions}, and a fourth operation, $N$, is a \emph{singularity switch}. The operations are defined as follows.
\begin{enumerate}
\item $\widetilde M_1$ transforms $(G,s,d,b)$ into the SBM $(G_1=G\amalg \{g\},s,d, b_1)$ such that $b_1:G_1 \times G_1 \rightarrow H$ extends $b$ and $b_1(g,h)=0$ for all $h\in G_1$.
\item $\widetilde M_2$ transforms $(G,s,d,b)$ into the SBM $(G_2=G\amalg \{g\},s,d, b_2)$ such that $b_2:G_2 \times G_2 \rightarrow H$ extends $b$ and $b_2(g,h)=b_2(s,h)$ for all $h\in G_2$.
\item $\widetilde M_3$ transforms $(G,s,d,b)$ into the SBM $(G_3=G\amalg \{g_i,g_j\},s, d, b_3)$ such that $b_3:G_3 \times G_3 \rightarrow H$ is any skew-symmetric map extending $b$ with $b_3(g_i,h)+b_3(g_j,h)=b_3(s,h)$ for all $h\in G_3$.
\item Suppose there exists an element $g \in G$ such that $b(g,h)+b(d,h)=b(s,h)$ for all $h\in G$ (i.e., $g$ and $d$ are complementary). Then $N$ transforms $(G,s,d,b)$ into $(G,s,g,b)$. In effect, the roles of $g$ and $d$ are switched. \end{enumerate}

We will denote by $\widetilde M_i^{-1}$ the inverse operation for $\widetilde M_i$. We can see that, for $i=1,2$, $\widetilde M_i^{-1}$ and $\widetilde M_i$ are in fact only inverses up to isomorphism, where isomorphism has the meaning given below. What's more, $\widetilde M_3^{-1}\circ \widetilde M_3$ may not even be equivalent to an isomorphism. 

\begin{Def} \begin{enumerate} 
\item Two SBMs $(G, s,d, b)$ and $(G',s',d',b')$ are \emph{isomorphic} if there is a bijection $G\rightarrow G'$ sending $s$ to $s'$, $d$ to $d'$ and transforming $b$ into $b'$.
\item A SBM, $(G,s,d,b)$, is \emph{primitive} if it cannot be obtained from another SBM by an elementary extension, even after applications of the singularity switch operation.
\item Two SBMs are \emph{homologous} if one can be obtained from the other by a finite number of elementary extensions $\widetilde M_1^{\pm 1},\widetilde M_2^{\pm 1},\widetilde M_3^{\pm 1}$ and $N$ moves.
\end{enumerate}
\end{Def}

Now that the notion of isomorphism has been explicitly defined, we are equipped to show that $N$ has the following useful property.

\begin{Prop}
The singularity switch, $N$, is its own inverse, up to an isomorphism.
\end{Prop}

\begin{proof}
Consider the sequence of moves $N\circ N$. The first instance of $N$ switches $d$ with a complementary element $g$, thus replacing the SBM $(G,s,d,b)$ with $(G,s,g,b)$. Now $g$ is complementary to $d$, so if $N$ is to switch $g$ with an element $g'$, then $g'$ must satisfy the property that $b(g',h)=b(d,h)$ for all $h\in G$. Because $g'$ and $d$ agree with respect to $b$, there is an isomorphism sending $g'$ to $d$ and $d$ to $g'$. Thus, $N\circ N$ is  an isomorphism.
\end{proof}

Turaev showed in~\cite{Turaev} that there is a unique primitive based matrix (up to isomorphism) in each homology class. While the same result doesn't generally hold for SBMs, we can prove a similar result about how the primitive SBMs are related to one another in a given homology class. Let us first consider the results that \emph{do} carry over from the based matrix setting to the SBM setting.

\begin{Lemma} Every SBM is obtained from a primitive SBM by a finite number of elementary extensions and singularity switches.\end{Lemma}
\begin{proof} Given a SBM $T=(G,s,d,b)$, we can reduce the cardinality of $G$ by eliminating any core elements, annihilating elements and complementary pairs with $\widetilde M_1^{- 1},\widetilde M_2^{- 1},\widetilde M_3^{- 1}$. If no such eliminations are possible, we can see if $N$ applies. If so, we perform the $N$ move and check if any $M_i^{-1}$ can be applied to the result. Proceeding in this fashion, since the size of $G$ is monotonically decreasing, the process must eventually terminate. The resulting SBM will clearly be primitive. \end{proof}

For the next theorem, we use the following shorthand. Let $D_{21}$ denote the sequence $\widetilde M_2^{-1}\circ N\circ \widetilde M_1$ that adds an annihilating element, exchanges it with $d$ (assuming $d$ is core-like), and removes the new core element. In essence, the move $D_{21}$ is the move that replaces a core-like $d$ by an annihilating-like $d$.  $D_{12}$ will denote the inverse operation $\widetilde M_1^{-1}\circ N\circ \widetilde M_2$ that adds a core element, exchanges it with $d$ and deletes the new annihilating element. Again, we simply view $D_{12}$ as a move that replaces an annihilating-like $d$ with a core-like $d$.\\

\begin{thm}\label{annoyingthm} Given two homologous primitive SBMs, the second can be obtained from the first by an isomorphism or a composition of an isomorphism with a single $D_{12}$, $D_{21}$ or $N$ move. \end{thm}

\begin{proof} Let $P_{\bullet}$ and $P_{\bullet}'$ be two homologous primitive SBMs. Since they are homologous, $P_{\bullet}$ and $P_{\bullet}'$ are related by a sequence of isomorphisms, $\widetilde M_i^{\pm1}$ moves and $N$ moves. In~\cite{Turaev}, Turaev showed that for any $i,j\in\{1,2,3\}$, $M_j^{-1}\circ M_i$ can be replaced by an isomorphism or a sequence $M_l\circ M_k^{-1}$for some $k,l\in\{1,2,3\}$. The same result (with the same proof) holds if each $M_i$ is replaced by $\widetilde M_i$. Furthermore, an isomorphism composed with an elementary extension or singularity switch may be replaced by the elementary extension or singularity switch composed with an isomorphism in the reverse order.
Our strategy, then, will be to build on Turaev's results and show that the sequence of moves required to obtain $P_{\bullet}'$ from $P_{\bullet}$ can be replaced by a sequence with the following form. 

\begin{equation}\label{DesiredSequence}
(\widetilde M_i\text{ and }N\text{ moves})\circ (\text{one of }D_{12},D_{21}\text{ and }N)\circ (\text{isomorphisms})\circ (\widetilde M_i^{-1}\text{ and }N\text{ moves})
\end{equation}

Since $P_{\bullet}$ and $P_{\bullet}'$ are primitive, this would imply the theorem. Indeed if $P_{\bullet}$ is primitive, no sequence of $\widetilde M_i^{-1}$ and $N$ moves can be applied to $P_{\bullet}$. Similarly, $P_{\bullet}'$ cannot be obtained from a sequence of $\widetilde M_i$ and $N$ moves. So by the primitivity of the SBMs, any sequence that has the form stated above is in fact a sequence of the following form.

 $$(\text{one of }D_{12},D_{21}\text{ and }N)\circ (\text{isomorphisms})$$
 
Let us show that any sequence consisting entirely of $D_{ij}$ and $N$ moves can be replaced by a sequence containing at most one instance of $D_{ij}$ or $N$. Consider $D_{ij}\circ N$ where $i=1,j=2$, then $d$ must be a core-like element which $N$ switches with an annihilating element. After $D_{12}$, we have $d$ back to it's original (core) state, so the only change is that our original annihilating element has been replaced by a core element as a result of the first $N$ move. Thus, $D_{12}\circ N$can be replaced by $\widetilde M_2\circ \widetilde M_1^{-1}$. Similarly, $D_{21}\circ N$ can be replaced by $\widetilde M_1\circ \widetilde M_2^{-1}$, while $N\circ D_{ij}$ can also be replaced by $\widetilde M_j\circ \widetilde M_i^{-1}$. Let us further note that $D_{12}$ and $D_{21}$ are inverses while $N$ is its own inverse, up to isomorphism. For $i,j=1,2$, the sequence $D_{ij}\circ D{ij}$ cannot occur unless the core element is annihilating-like, in which case $i$ and $j$ are interchangeable. Thus, any sequence containing two or more consecutive instances of any of $D_{12},D_{21},N$ can be replaced by a sequence containing at most one of the three moves.\\

Now we turn to the interaction of $D_{ij}$ moves with $\widetilde M_k^{\pm1}$ moves. Note that in the sequence $\widetilde M_k^{-1}\circ D_{ij}$, the move $\widetilde M_k^{-1}$ can't involve elements added or deleted in $D_{ij}$, so the move is equivalent to $D_{ij}\circ \widetilde M_k^{-1}$. Similarly, $D_{ij}\circ \widetilde M_k$ is equivalent to $\widetilde M_k\circ D_{ij}$. Thus, $D_{ij}$ moves commute with $\widetilde M_1^{\pm 1},\widetilde M_2^{\pm 1},\text{ and }\widetilde M_3^{\pm 1}$.\\

It remains to investigate sequences of moves involving only elementary extensions and singularity switches. It is our goal to show that an arbitrary sequence involving these types of moves may be replaced by a sequence with the preferred ordering indicated in \eqref{DesiredSequence}. Since cases involving only elementary extensions have already been considered in Turaev's work, we consider strings of moves of the form $ \widetilde M_j^{-1}\circ N\circ \widetilde M_i$ that \emph{aren't} equivalent to $D_{12},D_{21}$ or an isomorphism. We would like to show that these sequences of moves can be replaced by sequences of the same or shorter length such that any inverse extensions are performed before any extensions. Thus by induction, we will have achieved our desired result.

\begin{itemize}
\item \emph{Case $i=1,j=1$.} Suppose that $N$ involves elements affected by both $\widetilde M_1$ and $\widetilde M_1^{-1}$ moves. Then $d$ must be core-like when $N\circ \widetilde M_1$ is applied. After these first two moves, we have a new core element coming from the old $d$, and the new $d$ is annihilating-like. But if $M_1^{-1}$ can be applied to our new core element, $s$ must be annihilating-like. Thus $M_1=M_2$, so the $\widetilde M_1^{-1}\circ N\circ \widetilde M_1$ sequence is equivalent to $D_{12}$, contradicting our assumption. Next, suppose that $N$ does not use the new element introduced by $\widetilde M_1$, then we can replace the sequence $\widetilde M_1^{-1}\circ N\circ \widetilde M_1$ by $\widetilde M_1^{-1}\circ \widetilde M_1\circ N$, which in turn can be replaced by $N$ (and possibly an isomorphism). Finally, suppose that $N$ doesn't involve the element removed by $\widetilde M_1^{-1}$.  Then $\widetilde M_1^{-1}$ commutes with $N$. So $\widetilde M_1^{-1}\circ N\circ \widetilde M_1$ can be replaced by $N\circ \widetilde M_1^{-1}\circ \widetilde M_1$, which in turn can be replaced by $N$.
\item \emph{Case $i=2,j=2$.} This is similar to the previous case.
\item \emph{Case $i=1,j=2$.} Since it is assumed that $\widetilde M_2^{-1}\circ N\circ \widetilde M_1$ isn't equivalent to $D_{21}$, $N$ can involve at most one of the elements of $G$ affected by the $\widetilde M_1$ and $\widetilde M_2^{-1}$ moves. So either $\widetilde M_1$ commutes with $N$ or $\widetilde M_2^{-1}$ commutes with $N$. If we are in the first situation, we can replace the sequence $\widetilde M_2^{-1}\circ N\circ \widetilde M_1$ by $\widetilde M_2^{-1}\circ \widetilde M_1\circ N$, which in turn can be replaced by $\widetilde M_l\circ \widetilde M_k^{-1}\circ N$ for some $k,l\in\{1,2,3\}$. In the second situation, $\widetilde M_2^{-1}\circ N\circ \widetilde M_1$ can be replaced by $N\circ \widetilde M_2^{-1}\circ \widetilde M_1$, which in turn can be replaced by $N\circ \widetilde M_l\circ\widetilde M_k^{-1}$ for some $k,l\in\{1,2,3\}$.
\item \emph{Case $i=2,j=1$.} This is similar to the previous case.
\item \emph{Case $i=1,j=3$.} Suppose $N$ involves elements from both $\widetilde M_1$ and $\widetilde M_3^{-1}$ moves. Then $d$ must be core-like and be switched with the new annihilating element. Furthermore, $\widetilde M_3^{-1}$ must delete a core element and an annihilating element. Because the end result is that we've switched an existing annihilating element with $d$ and deleted a core element, this sequence can be replaced by $\widetilde M_2^{-1}\circ N$. Now if $N$ does not involve elements from both $\widetilde M_1$ and $\widetilde M_3^{-1}$ moves, then $N$ commutes with one of the two other moves. So by Turaev's result, $\widetilde M_3^{-1}\circ N\circ\widetilde M_1$ can be replaced by $N\circ \widetilde M_l\circ \widetilde M_k^{-1}$ or $\widetilde M_l\circ\widetilde M_k^{-1}\circ N$ for some $k,l\in\{1,2,3\}$.
\item \emph{Case $i=2,j=3$.} This is similar to the previous case.
\item \emph{Case $i=3,j=1$.} Suppose $N$ involves elements from both $\widetilde M_3$ and $\widetilde M_1^{-1}$ moves. Then $\widetilde M_3$ must add a core element and an annihilating element and $d$ must be annihilating-like. After the sequence of moves, $d$ is core-like and there is an additional annihilating element. So the sequence $\widetilde M_1^{-1}\circ N\circ\widetilde M_3$ can be replaced by $\widetilde M_1\circ D_{12}$. If $N$ does not involve elements from both $\widetilde M_3$ and $\widetilde M_1^{-1}$ moves, then $N$ commutes with one of the two other moves. So by Turaev's result, $\widetilde M_1^{-1}\circ N\circ\widetilde M_3$ can be replaced by
$N\circ \widetilde M_l\circ \widetilde M_k^{-1}$ or $\widetilde M_l\circ\widetilde M_k^{-1}\circ N$ for some $k,l\in\{1,2,3\}$.
\item \emph{Case $i=3,j=2$.} This is similar to the previous case.
\item \emph{Case $i=3,j=3$.} Suppose $N$ involves elements from both $\widetilde M_3$ and $\widetilde M_3^{-1}$ moves. Then $\widetilde M_3$ adds a pair of complementary elements, one of which behaves like $d$ with respect to the $b$ map. After the switch, $d$ has $b$ values complementary to its original values and there are two elements in $G-\left\{s,d\right\}$ that have identical $b$ values equal to the original values of $d$. In order for $\widetilde M_3^{-1}$ to involve one of these elements, there must be another element $g$ such that $b(d,h)=b(g,h)$ for all $h\in G$. So after $\widetilde M_3^{-1}$ is applied, the result is the same as if a simple $N$ move had been performed. Thus $\widetilde M_3^{-1}\circ N \circ\widetilde M_3$ can be replaced by $N$. Once again, if $N$ does not involve elements from both $\widetilde M_3$ and $\widetilde M_3^{-1}$ moves, then $N$ commutes with one of the two other moves. So using Turaev's result, $\widetilde M_3^{-1}\circ N\circ\widetilde M_3$ can be replaced by $N\circ \widetilde M_l\circ \widetilde M_k^{-1}$ or $\widetilde M_l\circ\widetilde M_k^{-1}\circ N$ for some $k,l\in\{1,2,3\}$.\\
\end{itemize} 

We have shown that our sequence can be reordered into the form we required above, thus the proof is complete.
\end{proof}

\begin{Cor} In each equivalence class of SBMs, there is a unique primitive SBM or a pair of primitive SBMs, up to isomorphism.
\end{Cor}
\begin{proof} Assuming that $D_{12}$ and $D_{21}$ are distinct operations (i.e., $s$ is not annihilating-like), it is impossible that more than one of the three moves $D_{12}$, $D_{21}$ or $N$ applies to a given primitive SBM, $P_{\bullet}$. This is because $d$ is either annihilating-like, core-like, or neither. Each of the moves can only be applied in one of these three situations. It may indeed happen that none of these operations applies to $P_{\bullet}$, rendering $P_{\bullet}$ the unique primitive SBM in its homology class (modulo isomorphism). Otherwise, the homology class of $P_{\bullet}$ contains a pair of primitive SBMs related by one of the three moves.\end{proof}

Now we return to the problem of distinguishing flat singular virtual knots with one double-point. First, we note that flat singular virtuals with one double-point can be viewed as virtual strings where one arrow is designated as the preferred arrow. This can be pictured by using a thickened arrow. We will refer to these modified virtual strings as \emph{singular virtual strings}. Naturally, the preferred arrow in a singular virtual string corresponds to the double-point in the flat singular virtual knot diagram. Therefore, the inverses of string moves (i) and (ii), which correspond to removing a kink with the Reidemeister 1 move, must not be allowed when one of the arrows involved is the preferred arrow. However, to model the flat version of the (S2) move pictured in Figure~\ref{SingEquiv}, we add a singular virtual string move that allows us to change which arrow is the preferred arrow in the diagram. Suppose $(a,b)$ is the preferred arrow in the diagram and let $(a',b')$ be an arrow in the diagram such that the interior of one of the arcs in the core circle $S$ with endpoints $a$ and $b'$ is disjoint from $arr(\alpha)$ and the interior of one of the arcs in $S$ with endpoints $a'$ and $b$ is disjoint from $arr(\alpha)$. Then the designation as the preferred arrow may switch from $(a,b)$ to $(a',b')$. Let us call this move (s-ii). Note that we still allow moves (i)-(iii) whenever they involve only ordinary arrows. We also allow move (iii), even if one of the arrows is the preferred one. The homotopy equivalence relation that results from this collection of moves corresponds precisely to the equivalence relation on flat singular virtual knots with one double point.\\

We may now use our theory of SBMs to define an invariant for singular virtual strings. Given a singular virtual string $\alpha$ with preferred arrow $(a,b)$, we let $G=G(\alpha)$ be the set $\{s\}\amalg arrow(\alpha)$ where $d$ in $(G,s,d,b)$ is precisely the preferred arrow $(a,b)$. The map $b$ is defined as it was for based matrices of strings, where $b(e,s)$ is the appropriate intersection index and $b(e,f)=e_1e_2\cdot f_1f_2-\epsilon$ where $s\neq e=(e_1,e_2), s\neq f=(f_1,f_2)$ and $\epsilon$ describes the linking of $e$ and $f$. We may picture the antisymmetric map $b$ as a matrix where the first row/column is the one corresponding to $s\in G$ and the last row/column is the one corresponding to $d$. Given this correspondence, we have the following result.

\begin{thm}\label{homo} If two singular virtual strings are homotopic, then their corresponding SBMs are homologous.
\end{thm}

\begin{proof} To show that the proposition holds, it suffices to check how applying moves (i)-(iii) and (s-ii) to a singular virtual string effects the corresponding SBMs. First, since $b$ is defined for singular virtual strings as it was for virtual strings, performing move (i) on a singular virtual string corresponds to applying $\widetilde M_1$ or $\widetilde M_2$ to the corresponding SBM, depending on the direction of the arrow that is added. Similarly, performing move (ii) on a singular virtual string precisely changes the corresponding SBM by an $\widetilde M_3$ transformation. Move (iii) applied to a singular virtual string does not effect the corresponding SBM, regardless of whether the (iii) move involves the preferred arrow. Finally, applying (s-ii) to a singular virtual string changes the corresponding SBM by a singularity switch, that is, by the transformation $N$. 
\end{proof}

Now that we have a tool for distinguishing singular virtual strings (i.e. flat singular virtual knots with one double-point), we can prove that the gluing invariant for virtual knots is stronger than the smoothing invariant. This proof will involve computations of values of the map $b$ for several SBMs. These computations will help to make the definition we introduced in the previous section more concrete.

\subsection{Proof of the second assertion of Theorem~\ref{stronger}}

To prove the second assertion of Theorem~\ref{stronger}, we must produce a pair of virtual knots $K_1$ and $K_2$ such that the smoothing invariant $\mathbf{S}$ gives the same value on the two knots, while the gluing invariant $\mathbf{G}$ gives two distinct values. In our proof, we let $K_1$ and $K_2$ be the knots pictured in Figure~\ref{strong}. Note that these knots are homotopic to one another, related by 2 (CC) moves applied to crossings 3 and 4. Thus, by showing that they are different using $\mathbf{G}$, we illustrate that their difference doesn't lie simply in the theory of homotopic virtual knots, but in their differences on the level of virtual knots. Of course, these same examples may also be distinguished from the unknot by $\mathbf{G}$ and not $\mathbf{S}$, but $K_1$ and $K_2$ were shown to be in non-trivial homotopy classes by Turaev in~\cite{Turaev}. Thus, they are both distinct from the unknot for more fundamental reasons. Let us now proceed to our proof.

To show that $\mathbf{S}(K_1)=\mathbf{S}(K_2)=0$, note that the link obtained by smoothing at crossing 1 in $K_1$ is exactly the same as the link obtained by smoothing at crossing 2. Likewise, the link obtained by smoothing at crossing 3 is the same link as the one obtained by smoothing at crossing 4. To show these links are equivalent, we need only apply Reidemeister moves (1) and (V1). Since crossings 1 and 2 have opposite signs and 3 and 4 have opposite signs, all terms in $\mathbf{S}(K_1)$ cancel. The same argument shows that $\mathbf{S}(K_2)=0$, so $\mathbf{S}(K_1)=\mathbf{S}(K_2)$.

\begin{figure}[h]
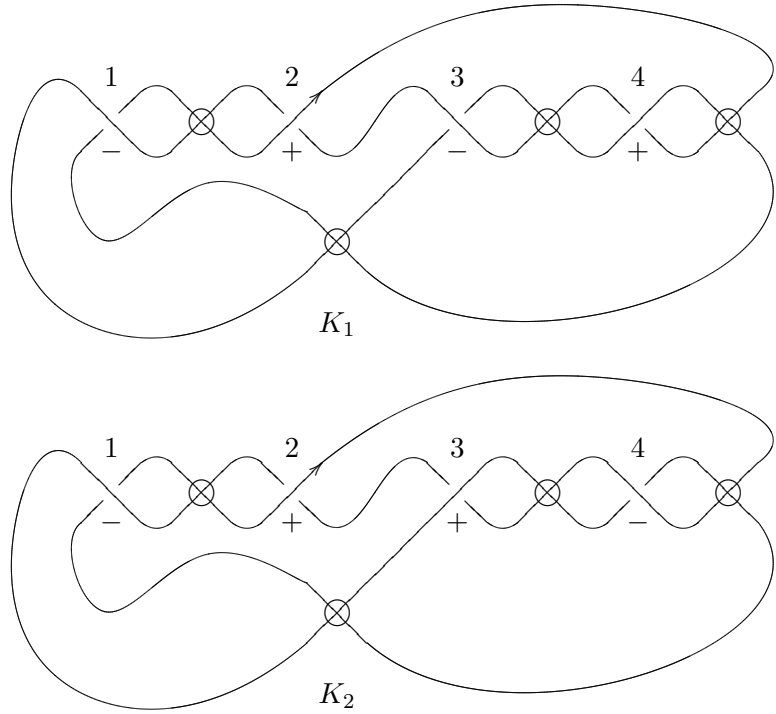

\[
\xy
(-50, -4)*{}="a1"; (-42,4)*{}="a2"; (-50,4)*{}="a3"; (-42,-4)*{}="a4";
(-38, -4)*{}="b1"; (-30,4)*{}="b2"; (-38,4)*{}="b3"; (-30,-4)*{}="b4";
(-26, -4)*{}="c1"; (-18,4)*{}="c2"; (-26,4)*{}="c3"; (-18,-4)*{}="c4";
(-4, -4)*{}="d1"; (4,4)*{}="d2"; (-4,4)*{}="d3"; (4,-4)*{}="d4";
(8, -4)*{}="e1"; (16,4)*{}="e2"; (8,4)*{}="e3"; (16,-4)*{}="e4";
(20, -4)*{}="f1"; (28,4)*{}="f2"; (20,4)*{}="f3"; (28,-4)*{}="f4";
(32, -4)*{}="g1"; (40,4)*{}="g2"; (32,4)*{}="g3"; (40,-4)*{}="g4";
(-20, -20)*{}="h1"; (-12,-20)*{}="h2"; (-20,-12)*{}="h3"; (-12,-12)*{}="h4";
"a3"; "a4" **\crv{}; \POS?(.5)*{\hole}="a"; ?(1)*\dir{}+(-4, 10)*{1}; ?(1)*\dir{}+(0, -4)*{-};
"a1"; "a" **\dir{-};
"a"; "a2" **\dir{-};
"c1"; "c2" **\crv{}; \POS?(.5)*{\hole}="c"; ?(1)*\dir{>}+(-4, 2)*{2}; ?(1)*\dir{}+(0, -4)*{+};
"c3"; "c" **\dir{-};
"c"; "c4" **\dir{-};
"d3"; "d4" **\crv{}; \POS?(.5)*{\hole}="d"; ?(1)*\dir{}+(-4, 10)*{3}; ?(1)*\dir{}+(0, -4)*{-};
"d1"; "d" **\dir{-};
"d"; "d2" **\dir{-};
"f1"; "f2" **\crv{}; \POS?(.5)*{\hole}="f"; ?(1)*\dir{}+(-4, 2)*{4}; ?(1)*\dir{}+(0, -4)*{+};
"f3"; "f" **\dir{-};
"f"; "f4" **\dir{-};
"b1"; "b2" **\dir{-};
"b3"; "b4" **\dir{-};
(-34,0)*{\bigcirc};
"e1"; "e2" **\dir{-};
"e3"; "e4" **\dir{-};
(12,0)*{\bigcirc};
"g1"; "g2" **\dir{-};
"g3"; "g4" **\dir{-};
(36,0)*{\bigcirc};
"h1"; "h4" **\dir{-};
"h2"; "h3" **\crv{}; \POS?(.5)*{}="h"; ?(1)*\dir{}+(4, -15)*{K_1};
(-16,-16)*{\bigcirc};
"a2"; "b3" **\crv{(-40,5.5)};
"b2"; "c3" **\crv{(-28,5.5)};
"a4"; "b1" **\crv{(-40,-5.5)};
"b4"; "c1" **\crv{(-28,-5.5)};
"d2"; "e3" **\crv{(6,5.5)};
"e2"; "f3" **\crv{(18,5.5)};
"d4"; "e1" **\crv{(6,-5.5)};
"e4"; "f1" **\crv{(18,-5.5)};
"f2"; "g3" **\crv{(30,5.5)};
"f4"; "g1" **\crv{(30,-5.5)};
"c4"; "d3" **\crv{(-15,-6)&(-11,0)&(-7,6)};
"h4"; "d1" **\crv{};
"c2"; "g2" **\crv{(-8,12) & (12,20) & (48, 10)};
"h2"; "g4" **\crv{(6,-36)&(52,-20)};
"a1"; "h3" **\crv{(-53,-6)&(-48,-24)& (-34,-2)&(-21,-12)};
"a3"; "h1" **\crv{(-54,8)&(-60,0)&(-60,-28)&(-34,-32)};
\endxy
\]
\[
\xy
(-50, -4)*{}="a1"; (-42,4)*{}="a2"; (-50,4)*{}="a3"; (-42,-4)*{}="a4";
(-38, -4)*{}="b1"; (-30,4)*{}="b2"; (-38,4)*{}="b3"; (-30,-4)*{}="b4";
(-26, -4)*{}="c1"; (-18,4)*{}="c2"; (-26,4)*{}="c3"; (-18,-4)*{}="c4";
(-4, -4)*{}="d1"; (4,4)*{}="d2"; (-4,4)*{}="d3"; (4,-4)*{}="d4";
(8, -4)*{}="e1"; (16,4)*{}="e2"; (8,4)*{}="e3"; (16,-4)*{}="e4";
(20, -4)*{}="f1"; (28,4)*{}="f2"; (20,4)*{}="f3"; (28,-4)*{}="f4";
(32, -4)*{}="g1"; (40,4)*{}="g2"; (32,4)*{}="g3"; (40,-4)*{}="g4";
(-20, -20)*{}="h1"; (-12,-20)*{}="h2"; (-20,-12)*{}="h3"; (-12,-12)*{}="h4";
"a3"; "a4" **\crv{}; \POS?(.5)*{\hole}="a"; ?(1)*\dir{}+(-4, 10)*{1}; ?(1)*\dir{}+(0, -4)*{-};
"a1"; "a" **\dir{-};
"a"; "a2" **\dir{-};
"c1"; "c2" **\crv{}; \POS?(.5)*{\hole}="c"; ?(1)*\dir{>}+(-4, 2)*{2}; ?(1)*\dir{}+(0, -4)*{+};
"c3"; "c" **\dir{-};
"c"; "c4" **\dir{-};
"d1"; "d2" **\crv{}; \POS?(.5)*{\hole}="d"; ?(1)*\dir{}+(-4, 2)*{3}; ?(1)*\dir{}+(0, -4)*{+};
"d3"; "d" **\dir{-};
"d"; "d4" **\dir{-};
"f3"; "f4" **\crv{}; \POS?(.5)*{\hole}="f"; ?(1)*\dir{}+(-4, 10)*{4}; ?(1)*\dir{}+(0, -4)*{-};
"f1"; "f" **\dir{-};
"f"; "f2" **\dir{-};
"b1"; "b2" **\dir{-};
"b3"; "b4" **\dir{-};
(-34,0)*{\bigcirc};
"e1"; "e2" **\dir{-};
"e3"; "e4" **\dir{-};
(12,0)*{\bigcirc};
"g1"; "g2" **\dir{-};
"g3"; "g4" **\dir{-};
(36,0)*{\bigcirc};
"h1"; "h4" **\dir{-};
"h2"; "h3" **\crv{}; \POS?(.5)*{}="h"; ?(1)*\dir{}+(4, -15)*{K_2};
(-16,-16)*{\bigcirc};
"a2"; "b3" **\crv{(-40,5.5)};
"b2"; "c3" **\crv{(-28,5.5)};
"a4"; "b1" **\crv{(-40,-5.5)};
"b4"; "c1" **\crv{(-28,-5.5)};
"d2"; "e3" **\crv{(6,5.5)};
"e2"; "f3" **\crv{(18,5.5)};
"d4"; "e1" **\crv{(6,-5.5)};
"e4"; "f1" **\crv{(18,-5.5)};
"f2"; "g3" **\crv{(30,5.5)};
"f4"; "g1" **\crv{(30,-5.5)};
"c4"; "d3" **\crv{(-15,-6)&(-11,0)&(-7,6)};
"h4"; "d1" **\crv{};
"c2"; "g2" **\crv{(-8,12) & (12,20) & (48, 10)};
"h2"; "g4" **\crv{(6,-36)&(52,-20)};
"a1"; "h3" **\crv{(-53,-6)&(-48,-24)& (-34,-2)&(-21,-12)};
"a3"; "h1" **\crv{(-54,8)&(-60,0)&(-60,-28)&(-34,-32)};
\endxy
\]

\caption{Examples illustrating that $\mathbf{G}$ is strictly stronger than $\mathbf{S}$}\label{strong}
\end{figure}

Now let us consider $\mathbf{G}(K_1)$. Since the writhe of the knot is 0, i.e. the sum of the signs of all the crossings is 0, all of the terms of the form $[(\widetilde K_1)_{sing}^0]$ in $\mathbf{G}(K_1)$ cancel. We are left with four terms. Let us consider the terms that correspond to crossing 3 and crossing 4. Crossing 3 contributes $(-1)[(\widetilde K_1)_{glue}^3]$, where this flat singular virtual knot associated to ``gluing" $K_1$ at crossing 3 can be viewed as the singular virtual string pictured on the left in Figure~\ref{stringterms}. Likewise, crossing 4 contributes $(+1)[(\widetilde K_1)_{glue}^4]$, where the flat singular virtual knot associated to $(K_1)_{glue}^4$ is the singular virtual string pictured on the right in Figure~\ref{stringterms}. We now compute the SBMs associated to these two singular virtual strings to show that the strings are distinct and, therefore, their terms in $\mathbf{G}(K_1)$ do not cancel one another as they did in $\mathbf{S}(K_1)$.

\def\TwoStringOne{
\begin{xy} /r15mm/:
,{\ellipse<>{}}
,(1.173648, .984808)="1" ,*+!D{d}
,(.826352,.984808)="2" ,*+!D{3}
,(1.173648, -.984808)="a1" 
,(.826352,-.984808)="a2"
,(1.984808, .173648)="3"
,(.015192, .173648)="a3" ,*+!R{2}
,(1.984808, -.173648)="4"
,(.015192, -.173648)="a4" ,*+!R{1}
,{\ar@{=>} "a1"; "1"}
,{\ar@{<-} "2"; "a2"}
,{\ar@{->} "3"; "a3"}
,{\ar@{->} "4"; "a4"}
\end{xy}
}

\def\TwoStringTwo{
\begin{xy} /r15mm/:
,{\ellipse<>{}}
,(1.173648, .984808)="1" ,*+!D{3}
,(.826352,.984808)="2" ,*+!D{d}
,(1.173648, -.984808)="a1" 
,(.826352,-.984808)="a2"
,(1.984808, .173648)="3"
,(.015192, .173648)="a3" ,*+!R{2}
,(1.984808, -.173648)="4"
,(.015192, -.173648)="a4" ,*+!R{1}
,{\ar@{->} "a1"; "1"}
,{\ar@{<=} "2"; "a2"}
,{\ar@{->} "3"; "a3"}
,{\ar@{->} "4"; "a4"}
\end{xy}
}

\begin{figure}[h]
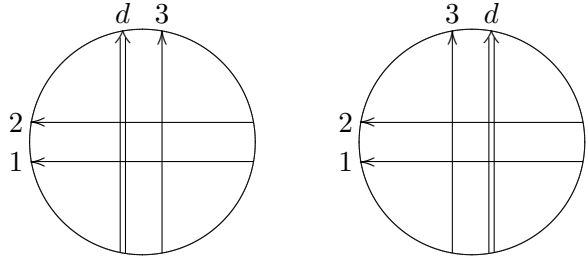

$$\TwoStringTwo\hspace{10mm}\TwoStringOne$$
\caption{Singular virtual strings associated to gluing $K_1$ at crossing 3 and gluing at crossing 4}\label{stringterms}
\end{figure}

Let $b_3$ be the skew-symmetric map in the SBM of the singular virtual string associated to crossing 3, as pictured on the left in Figure~\ref{stringterms}. (It may also be useful to think of this string as the flat singular virtual knot associated to gluing at crossing 3 in Figure~\ref{strong}.) By definition, $b_3(1,s)$ is the intersection index of the link obtained by smoothing at crossing 1 where we order the components so that the component on the right is component 1 and the component on the left is component 2. (Refer to Figure~\ref{Index} for the definition of intersection index and Figure~\ref{right} for an illustration of the ordering of components.) This index is easily seen to be -2, so $b_3(1,s)=-2$. Similar computations of intersection indices show that $b_3(2,s)=-2,$ $b_3(3,s)=2$ and $b_3(d,s)=2$. Next, let us consider $b_3(1,2)$. We will use the string notation $1=(a_1,b_1)$ for the arrow labeled 1 and $2=(a_2,b_2)$ for the arrow labeled 2. Then $b_3(1,2)=a_1b_1\cdot a_2b_2+\epsilon$. First of all, 1 and 2 are unlinked, so $\epsilon=0$. The quantity $a_1b_1\cdot a_2b_2$ is the number of arrows with tails in the interior of arc $a_1b_1$ and heads in the interior of arc $a_2b_2$ minus the number of arrows with tails in the interior of arc $a_2b_2$ and heads in the interior of arc $a_1b_1$. This quantity is 0-0=0. Thus, $b_3(1,2)=0$. To see a less trivial example, let us compute $b_3(1,3)$. Again, we use the notation $3=(a_3,b_3)$. Now, 3 links 1 negatively, so $\epsilon=-1$. Furthermore, the quantity $a_1b_1\cdot a_3b_3$ is $0-1=-1$, so $b_3(1,3)=a_1b_1\cdot a_3b_3+\epsilon=-1+(-1)=-2$. Proceeding in this way, we find that $b_3$ takes the following values: $b_3(1,d)=-3$, $b_3(2,3)=-1$, $b_3(2,d)=-2$, $b_3(3,d)=0$. We organize these values in a matrix as follows.

\begin{displaymath}
\begin{array}{cc}
&\begin{array}{ccccccc} &s   & 1  & 2  & 3  & d &\end{array}\\
\begin{array}{c} s\\ 1\\ 2\\ 3\\ d\end{array} &
\left[ \begin{array}{ccccc} 
0 & 2 & 2 & -2 & -2\\ 
-2 & 0 & 0 & -2 & -3\\
-2 & 0 & 0 & -1 & -2\\
2 & 2 & 1 & 0 & 0\\
2 & 3 & 2 & 0 & 0
\end{array}\right] 
\end{array}
\end{displaymath}

The computations of $b_4$ for the SBM of the singular virtual string associated to crossing 4 are similar. We represent them here in matrix form.

\begin{displaymath}
\begin{array}{cc}
&\begin{array}{ccccccc} &s   & 1  & 2  & 3  & d &\end{array}\\
\begin{array}{c} s\\ 1\\ 2\\ 3\\ d\end{array} &
\left[ \begin{array}{ccccc} 
0 & 2 & 2 & -2 & -2\\ 
-2 & 0 & 0 & -3 & -2\\
-2 & 0 & 0 & -2 & -1\\
2 & 3 & 2 & 0 & 0\\
2 & 2 & 1 & 0 & 0
\end{array}\right] 
\end{array}
\end{displaymath}

We note that these matrices are both primitive. Furthermore, $d$ is clearly not annihilating-like or core-like in either matrix, and the two matrices are not related by an isomorphism or an $N$ move. Thus, by Theorem~\ref{annoyingthm}, they are not homologous.  It follows from Proposition~\ref{homo}, then, that the two singular virtual strings in Figure~\ref{stringterms} are not homotopic. Thus, the terms in $\mathbf{G}(K_1)$ corresponding to crossings 3 and 4 do not cancel as they did in $\mathbf{S}(K_1)$.

We see that the terms in $\mathbf{G}(K_2)$ corresponding to crossings 3 and 4 in $K_2$ are the same as in $\mathbf{G}(K_1)$, except with opposite sign. Thus, $\mathbf{G}(K_1)-\mathbf{G}(K_2)$ is non-zero. Indeed this difference has two terms, one with coefficient +2 and one with coefficient -2. By the analysis above, these terms do not cancel and, hence, $\mathbf{G}(K_1)\neq\mathbf{G}(K_2)$. \emph{Q.E.D.}

\bibliographystyle{amsplain}
\bibliography{Preprint}

\end{document}